\documentclass[reqno, 12pt]{amsart}

\usepackage{caption}
\usepackage{array}
\usepackage{amsmath}
\usepackage{amsfonts}
\usepackage{amssymb}
\usepackage{enumerate}
\usepackage{epsfig}
\usepackage{graphicx}
\usepackage{amsthm}
\usepackage{amsmath, amscd}
\usepackage{xy}
\xyoption{all}
\usepackage{color}
\usepackage{marvosym}
\usepackage{bm}
\usepackage{amsbsy}
\usepackage{hyperref}
\usepackage{tikz}
\usetikzlibrary{matrix,arrows,backgrounds}

\newtheorem{theorem}{Theorem}[subsection]
\newtheorem{theorem-section}{Theorem}[section]

\newtheorem{lemma}[theorem]{Lemma}
\newtheorem{proposition}[theorem]{Proposition}
\newtheorem{proposition-section}[theorem-section]{Proposition}

\newtheorem{corollary-section}[theorem-section]{Corollary}

\theoremstyle{definition}
\newtheorem{definition}[theorem]{Definition}
\newtheorem{remark}[theorem]{Remark}
\newtheorem{remark-section}[theorem-section]{Remark}

\newcommand{\op}[1]{\operatorname{#1}}

\newcommand{\leftexp}[2]{{\vphantom{#2}}^{#1}{#2}}

\newcommand{\dcoh}[1]{\operatorname{D}(\operatorname{coh }#1)}
\newcommand{\dbcoh}[1]{\operatorname{D}^{\operatorname{b}}(\operatorname{coh }#1)}

\newcommand{\dperf}{\operatorname{D}^{\operatorname{pe}}}

\def\Z{\op{\mathbb{Z}}}

\def\O{\op{\mathcal{O}}}

\def\P{\op{\mathbb{P}}}

\def\ra{\rightarrow}

\def\tif{\text{if } }

\def\tand{\text{ and } }

\def\GG{\op{\mathbb{G}}}

\def\cX{\mathcal{X}}
\def\cL{\mathcal{L}}

\def\cE{\mathcal{E}}
\def\cA{\mathcal{A}}
\def\cB{\mathcal{B}}

\def\cO{\mathcal{O}}
\def\cM{\mathcal{M}}

\def\cV{\mathcal{V}}
\def\cW{\mathcal{W}}
\def\cC{\mathcal{C}}
\def\cU{\mathcal{U}}

\def\cP{\mathcal{P}}
\def\cT{\mathcal{T}}
\def\cZ{\mathcal{Z}}
\def\cD{\mathcal{D}}

\def\gm{\mathbb{G}_m}

\def\wtQ{\widetilde{Q}_{\mathcal V}}

\def\iso{\cong}
\def\cbullet{\raisebox{0.02cm}{\scalebox{0.7}{$\bullet$}}}

\newcommand{\newterm}{\textsf}

\title[HPD via VGIT]{Homological Projective Duality via Variation of Geometric Invariant Theory Quotients}
\author[Ballard]{Matthew Ballard}
\address{
  \begin{tabular}{l}
   Matthew Ballard  \\ 
   \hspace{.1in} Universit\"at Wien, Fakult\"at f\"ur Mathematik,  Wien, \"Osterreich \\
   \hspace{.1in} Email: {\bf ballard@math.wisc.edu} \\
  \end{tabular}
}

\author[Deliu]{Dragos Deliu}
\address{
  \begin{tabular}{l}
   Dragos Deliu  \\ 
   \hspace{.1in} Universit\"at Wien, Fakult\"at f\"ur Mathematik,  Wien, \"Osterreich \\
   \hspace{.1in} Email: {\bf dragos.deliu@univie.ac.at} \\
  \end{tabular}
}

\author[Favero]{David Favero}
\address{
  \begin{tabular}{l}
   David Favero \\
   \hspace{.1in} Universit\"at Wien, Fakult\"at f\"ur Mathematik,  Wien, \"Osterreich \\
   \hspace{.1in} Email: {\bf favero@gmail.com} \\
  \end{tabular}
}

\author[Isik]{M. Umut Isik}
\address{
  \begin{tabular}{l}
   M. Umut Isik \\ 
   \hspace{.1in} Universit\"at Wien, Fakult\"at f\"ur Mathematik,  Wien, \"Osterreich \\
   \hspace{.1in} Email: {\bf mehmet.umut.isik@univie.ac.at} \\
  \end{tabular}
}

\author[Katzarkov]{Ludmil Katzarkov}
\address{
  \begin{tabular}{l}
   Ludmil Katzarkov \\
   \hspace{.1in} Universit\"at Wien, Fakult\"at f\"ur Mathematik,  Wien, \"Osterreich \\
   \hspace{.1in} Email: {\bf lkatzark@math.uci.edu} \\
  \end{tabular}
}

\numberwithin{equation}{section}

\begin{document}
\begin{abstract}
We provide a geometric approach to constructing Lefschetz collections and Landau-Ginzburg Homological Projective Duals from a variation of Geometric Invariant Theory quotients. This approach yields homological projective duals for Veronese embeddings in the setting of Landau Ginzburg models. Our results also extend to a relative Homological Projective Duality framework.  
\end{abstract}
\maketitle

\section{Introduction} \label{sec: intro}

A fundamental question in algebraic geometry is how invariants behave under passage to hyperplane sections.  In his seminal work \cite{KuzHPD}, Kuznetsov studied this question extensively for the bounded derived category of coherent sheaves on a projective variety and developed a deep homological manifestation of projective duality.   He suitably titled this phenomenon, ``Homological Projective Duality'' (HPD).  

The HPD setup is as follows.  One starts with a smooth variety $X \rightarrow \P(V)$, together with some homological data which is called a Lefschetz decomposition, and constructs a Homological Projective Dual; $Y\rightarrow \P(V^*)$  together with a dual Lefschetz decomposition. 
This establishes a precise relationship between the derived categories of any dual complete linear sections $X \times_{\P(V)} \P(L^\perp)$ and $Y \times_{\P(V^*)} \P(L)$; we call this result of Kuznetsov the ``Fundamental Theorem of Homological Projective Duality.''  \cite[Theorem 6.3]{KuzHPD} (Theorem \ref{thm: fundamental HPD} below). 

In this paper, we develop a robust geometric approach to constructing Homological Projective Duals as Landau-Ginzburg models.  The idea, in the terminology of high-energy theoretical physics, is to pass to a gauged linear sigma model and ``change phases" \cite{HHP, HTo, DS, CDHPS, Sha}.  In mathematical terms, this is first passing from a hypersurface to the total space of a line bundle \cite{Isik, Shipman}, then varying Geometric Invariant Theory quotients (VGIT) to do a birational transformation to the total space of this line bundle \cite{BFK12, HL12, KawFF,VdB, Seg2, HW, DSe}.  A nice consequence of our technique is that we can expand the Homological Projective Duality framework to the relative setting i.e.\ all our results are proven in the relative setting over a general smooth base variety.

Specifically, using the semi-orthogonal decompositions from \cite{BFK12}, we construct both Lefschetz collections and Homological Projective duals for a large class of quotient varieties. Our main application is to Veronese embeddings $\P(W) \rightarrow \P(S^dW)$ for $d\leq \op{dim}W$. After recovering the natural Lefschetz decomposition in this case, we prove that the Landau-Ginzburg pair $\left([ W\times \P(S^dW^*) / \mathbb{G}_m],w\right) $, where the $\mathbb{G}_m$-action is by dilation on $W$ and $w$ is the universal degree $d$ polynomial, is a homological projective dual to the Veronese embedding. 
 In a subsequent paper, \cite{BDFIKdegreed}, we replace the pair $\left([ W\times \P(S^dW^*) / \mathbb{G}_m],w\right)$ with the non-commutative space, $(\P(S^d W^*),\cA)$, where $\cA$ is a $\Z$-graded sheaf of minimal $A_\infty$-algebras given by
 $$\mathcal A=\left(\bigoplus_{k\in \Z} u^k\cO_{\P(S^d W^*)}(k) \right)\otimes\Lambda^{\bullet} W^*,$$
 and higher products defined by explicit tree formulas, notably with 
 $$\mu^d(1\otimes v_{i_1},\ldots,1\otimes v_{i_d})=\frac{u}{d!}\frac{\partial^d w}{\partial x_{i_1}\ldots\partial x_{i_d}},$$
 where $\left\{ x_j \right\}$ denotes a basis of $W$ and $\left\{ v_j \right\}$ the corresponding basis of $\Lambda^{1} W^*$,
and $\mu^i = 0$ for $2<i<d$. When $d=2$, this recovers the homological projective dual from \cite{Kuz05}. 


It should be noted that neither Kuznetsov's precise definition of a homological projective dual nor his Fundamental Theorem are available at this level of generality. We instead construct Landau-Ginzburg models which are \emph{weak} homological projective duals and prove that the conclusions of the Fundamental Theorem hold directly in our setting (Theorem \ref{thm: MainHPDMF}).


Homological Projective Duality was exhibited by Kuznetsov for the double Veronese embedding
\[
 \P(W)\hookrightarrow \P(S^2W).
\]
In this case, Kuznetsov \cite{Kuz05} proves that a homological projective dual is given by $Y=(\P(S^2W^*),\op{Cliff}_0)$, where $\op{Cliff}_0$ is a sheaf of even Clifford algebras. 
As a consequence, Kuznetsov recovers a theorem of Bondal and Orlov \cite{BO95} relating the derived category of intersection of two even dimensional quadrics to the derived category of a hyperelliptic curve. Moreover, his Homological Projective Duality framework provides analogous descriptions for arbitrary intersections of quadrics  as in \cite{BOICM}.

In \cite{Kuz06}, Kuznetsov constructs the dual to the Grassmannian of two dimensional planes in a vector space $W$ of dimension $6$ or $7$ with respect to the Pl\"ucker embedding,
\[
 \op{Gr}(2,W) \hookrightarrow \P(\Lambda^2W).
\]
In these cases, the homological projective dual is a non-commutative resolution of the classical projective dual: the Pfaffian variety. Among the many applications is a derived equivalence between two non-birational Calabi-Yau varieties of dimension $3$, originally studied by R\o{}dland as an example in mirror symmetry \cite{Rod}. This derived equivalence  was proven independently by Borisov and C\u{a}ld\u{a}raru \cite{BC} who demonstrated that generic Grassmannian Calabi-Yau varieties can be realized as moduli spaces of curves on the dual Pfaffian Calabi-Yau. Homological Projective Duality for the Grassmannian $\op{Gr}(3,6)$ was studied in \cite{deliu}.

A relative version of the $2$-Veronese example was considered in \cite{ABB}; it was used to relate rationality questions to categorical representability. Another example of Homological Projective Duality is conjectured by Hosono and Takagi and supported by a proof of a derived equivalence between the corresponding linear sections \cite{HTa1, HTa2, HTa3}.

The main cases that this paper does not interpret in our larger framework are the Grassmannian and Hosono-Takagi examples \cite{Kuz06, deliu, HTa1,HTa2, HTa3}.  However, these examples do admit similar physical interpretations  \cite{DS, Hori}.  Thus, it is plausible that all known examples of HPD would fall within the scope of our methodological approach.  The main issue is that the results of \cite{BFK12} need to be expanded to handle the complexity of the VGIT theory which arises.  Indeed, work of Addington, Donovan, and Segal, \cite{ADS} uses more complex GIT stratifications to understand the Grassmannian case, albeit in a slightly less general context than HPD.

\subsection*{Acknowledgments} We thank David Ben-Zvi, Colin Diemer, Alexander Efimov, Kentaro Hori, Maxim Kontsevich, Alexander Kuznetsov, Yanki Lekili, Dmitri Orlov, Pranav Pandit, Tony Pantev, Alexander Polishchuk and Anatoly Preygel for useful discussions. The authors were funded by NSF DMS 0854977 FRG, NSF DMS 0600800, NSF DMS 0652633 FRG, NSF DMS 0854977, NSF DMS 0901330, FWF P 24572 N25, by FWF P20778 and by an ERC Grant. The first author was funded, in addition, by NSF DMS 0838210 RTG.

\section{Background}\label{sec: background}

\subsection{Derived categories of LG models} \label{sec: LG}

Let $Q$ be a smooth and quasi-projective variety with the action of an affine algebraic group, $G$. Let $\cL$ be an invertible $G$-equivariant sheaf on $Q$ and let $w \in \op{H}^0(Q,\mathcal L)^G$ be a $G$-invariant section of $\cL$. We start by recalling the appropriate analog of the bounded derived category of coherent sheaves for a quadruple, $(Q, G, \mathcal L,w)$. Matrix factorization categories have been studied in \cite{EisMF, Buc86, Orl04}. Building on these works, most of the ideas presented here are due to L. Positselski \cite{Pos1,Pos2}. The authors generalize these ideas in \cite{BDFIK} to a setting which includes the material presented below.

\begin{definition}\label{deflgmodel}
 A \newterm{gauged Landau-Ginzburg model}, or \newterm{gauged LG model}, is the quadruple, $(Q, G, \cL, w)$, with $Q$, $G$, $\cL$, and $w$ as above. We shall commonly denote a gauged LG model by the pair $([Q/G],w)$. 
\end{definition}

To declutter the notation, given a quasi-coherent $G$-equivariant sheaf, $\mathcal E$, we denote $\mathcal E \otimes \cL^n$ by $\mathcal E(n)$. Given a morphism, $f: \mathcal E \to \mathcal F$, we denote $f \otimes \op{Id}_{\cL^n}$ by $f(n)$. Following Eisenbud, \cite{EisMF}, one gives the following definition.

\begin{definition}
 A \newterm{coherent factorization}, or simply a \newterm{factorization}, of a gauged LG model, $([Q/G],w)$, consists of a pair of coherent $G$-equivariant sheaves, $\mathcal E^{-1}$ and $\mathcal E^0$, and a pair of $G$-equivariant $\mathcal O_Q$-module homomorphisms,
\begin{align*}
 \phi^{-1}_{\mathcal E} &: \mathcal E^{0}(-1) \to \mathcal E^{-1} \\
 \phi^0_{\mathcal E} &: \mathcal E^{-1} \to \mathcal E^0
\end{align*}
such that the compositions, $\phi^0_{\mathcal E} \circ \phi^{-1}_{\mathcal E} : \mathcal E_{0}(-1) \to \mathcal E^0$ and  $\phi^{-1}_{\mathcal E}(1) \circ \phi_{\mathcal E}^0: \mathcal E^{-1} \to \mathcal E^{-1}(1)$, are multiplication by $w$.  We shall often simply denote the factorization $(\mathcal E^{-1}, \mathcal E^0, \phi_{\mathcal E}^{-1}, \phi_{\mathcal E}^0)$ by $\mathcal E$. The coherent $G$-equivariant sheaves, $\mathcal E^0$ and $\mathcal E^{-1}$, are called the \newterm{components of the factorization}, $\mathcal E$.

 A \newterm{morphism of factorizations}, $g: \mathcal E \to \mathcal F$, is a pair of morphisms of coherent $G$-equivariant sheaves,
 \begin{align*}
  g^{-1} & : \mathcal E^{-1} \to \mathcal F^{-1} \\
  g^0 & : \mathcal E^0 \to \mathcal F^0,
 \end{align*}
 making the diagram,
 \begin{center}
 \begin{tikzpicture}[description/.style={fill=white,inner sep=2pt}]
  \matrix (m) [matrix of math nodes, row sep=3em, column sep=3em, text height=1.5ex, text depth=0.25ex]
  {  \mathcal E^{0}(-1) & \mathcal E^{-1} & \mathcal E^{0} \\
   \mathcal F^{0}(-1) & \mathcal F^{-1} & \mathcal F^{0} \\ };
  \path[->,font=\scriptsize]
  (m-1-1) edge node[above]{$\phi^{-1}_{\mathcal E}$} (m-1-2) 
  (m-1-1) edge node[left]{$g^{0}(-1)$} (m-2-1)
  (m-2-1) edge node[above]{$\phi^{-1}_{\mathcal F}$} (m-2-2)
  (m-1-2) edge node[above]{$\phi^{0}_{\mathcal E}$} (m-1-3) 
  (m-1-2) edge node[left]{$g^{-1}$} (m-2-2)
  (m-2-2) edge node[above]{$\phi^{0}_{\mathcal F}$} (m-2-3) 
  (m-1-3) edge node[left]{$g^{0}$} (m-2-3)
  ;
 \end{tikzpicture} 
 \end{center}
 commute.
 
 We let $\op{coh}([Q/G],w)$ be the Abelian category of factorizations with coherent components.
 
 There is an obvious notion of a chain homotopy between morphisms in $\op{coh}([Q/G],w)$. Let $g_1,g_2 : \mathcal E \to \mathcal F$ be two morphisms of factorizations. A \newterm{homotopy} between $g_1$ and $g_2$ is a pair of morphisms of quasi-coherent $G$-equivariant sheaves,
 \begin{align*}
  h^{-1} & : \mathcal E^{-1} \to \mathcal F^{0}(-1) \\
  h^0 & : \mathcal E^0 \to \mathcal F^{-1},
 \end{align*}
 such that
 \begin{align*}
  g_1^{-1} - g_2^{-1} & = h^0 \circ \phi^0_{\mathcal E} + \phi^{-1}_{\mathcal F} \circ h^{-1} \\
  g_1^0 - g_2^0 & = h^{-1}(1) \circ \phi^{-1}_{\mathcal E}(1) + \phi^0_{\mathcal F} \circ h^0.
 \end{align*}
 
 We let $K(\op{coh}[Q/G],w)$ be the corresponding homotopy category, the category whose objects are factorizations and whose morphisms are homotopy classes of morphisms. 
\end{definition}

There is a translation autoequivalence, $[1]$, defined as 
\begin{displaymath}
 \mathcal E[1] := (\mathcal E^0, \mathcal E^{-1}(1), -\phi^0_{\mathcal E}, - \phi^{-1}_{\mathcal E}(1)).
\end{displaymath}
 
For any morphism, $g: \mathcal E \to \mathcal F$, there is a natural cone construction. We write, $C(g)$, for the resulting factorization. It is defined as 
\begin{displaymath}
C(g):= \left( \mathcal E^{0} \oplus \mathcal F^{-1}, \mathcal E^{-1}(1) \oplus \mathcal F^0, \begin{pmatrix} -\phi_{\mathcal E}^0 & 0 \\ g^{-1} & \phi_{\mathcal F}^{-1} \end{pmatrix}, \begin{pmatrix} -\phi_{\mathcal E}^{-1}(1) & 0 \\ g^0 & \phi_{\mathcal F}^0 \end{pmatrix} \right).
\end{displaymath}
The translation and the cone construction induce the structure of a triangulated category on the homotopy category, $K(\op{Qcoh}[Q/G],w)$ \cite{Pos2, BDFIK}.

We wish to derive $\op{coh}([Q/G],w)$, however, we lack a notion of quasi-isomorphism because our ``complexes'' lack cohomology. For the usual derived categories of sheaves, one can view localization by the class of quasi-isomorphisms as the Verdier quotient by acyclic objects. The correct substitute in $\op{coh}([Q/G],w)$ for acyclic complexes was defined independently in \cite{Pos1}, \cite{OrlMF}.

The following definitions give the correct analog for the derived category of sheaves for LG models, when $Q$ is smooth. These definitions are due to Positselski, see \cite{Pos1,Pos2}.

\begin{definition}
 A factorization, $\mathcal A$, is called \newterm{totally acyclic} if it lies in the smallest thick subcategory of $K(\op{coh}[Q/G],w)$ containing all totalizations \cite[Definition 2.10]{BDFIK} of short exact sequences 
from $\op{coh}([Q/G],w)$. We let $\op{acycl}([Q/G],w)$ denote the thick subcategory of $K(\op{coh}[Q/G],w)$ consisting of totally acyclic factorizations.  
 
 The \newterm{absolute derived category of factorizations}, or the \newterm{derived category}, of the LG model $([Q/G],w)$, is the Verdier quotient,
\begin{displaymath}
 \op{D}(\op{coh}[Q/G],w) := K(\op{coh}[Q/G],w)/\op{acycl}([Q/G],w).
\end{displaymath}
 
 Abusing terminology, we say that $\mathcal E$ and $\mathcal F$ are \newterm{quasi-isomorphic} factorizations if they are isomorphic in the absolute derived category. 
\end{definition}

Later in the paper we will also use the singularity category as an intermediary. We recall the definition.

\begin{definition}
 Let $[Y/G]$ be a global quotient stack with $Y$ quasi-projective. The \newterm{category of singularities} of $[Y/G]$ is the Verdier quotient,
 \begin{displaymath}
  \op{D}_{\op{sg}}([Y/G]) := \dbcoh{[Y/G]}/\op{perf}([Y/G]),
 \end{displaymath}
 of the bounded derived category of coherent sheaves by the thick subcategory of perfect complexes. 
\end{definition}

The following result, based on Koszul Duality, is referred to in the physics literature as the $\sigma$-model/Landau-Ginzburg-model correspondence for B-branes, arising from renormalization group flow. We sometimes refer to it briefly as the ``$\sigma$-LG correspondence''.

\begin{theorem}  \label{thm: isik}
 Let $Y$ be the zero-scheme of a section $s\in \Gamma(X,\cE)$ of a locally-free sheaf of finite rank $\cE$ on a smooth variety $X$. Assume that $s$ is a regular section, i.e. $\op{dim}Y = \op{dim}X - \op{rank} \cE$. Then, there is an equivalence of triangulated categories
 \begin{equation*}
  \dbcoh{Y} \iso \dcoh{[\op{V}(\cE)/\mathbb{G}_m], w}
 \end{equation*}
 where $\op{V}(\cE)$ is the total space $\mathbf{Spec}\operatorname{Sym}(\cE)$, $w$ is the regular function determined by $s$ under the natural isomorphism
 \begin{displaymath}
  \Gamma(\op{V}(\cE),\mathcal O) \cong \Gamma(X, \op{Sym} \mathcal E),
 \end{displaymath}
 and $\mathbb{G}_m$ acts by dilation on the fibers.
\end{theorem}

\begin{proof}
 This is \cite[Theorem 3.6]{Isik} or \cite[Theorem 3.4]{Shipman}.
\end{proof}

\begin{proposition} \label{cor: isik}
 Let $X$ be a smooth variety or a global quotient stack. Consider the trivial $\mathbb{G}_m$-action on $X$ and let $\chi$ be the identity character of $\mathbb{G}_m$. For the Landau-Ginzburg model $(X, \mathbb{G}_m, \O(\chi), 0)$, one has an equivalence of categories:
 \begin{equation*}
  \dbcoh{X} \iso \dcoh{[X/\mathbb{G}_m], 0}
 \end{equation*}
 \end{proposition}
 
\begin{proof}
  Consider the functor $F: \op{Kom}^{\op{b}}(\op{coh}X) \rightarrow \op{coh}(\left[ X/\gm \right],0)$ that sends a bounded complex $M^{\bullet}$ of coherent sheaves on $X$ to the factorization
\begin{center}
\begin{tikzpicture}[description/.style={fill=white,inner sep=2pt}]
\matrix (m) [matrix of math nodes, row sep=3em, column sep=3em, text height=1.5ex, text depth=0.25ex]
{  \bigoplus_i M^{2i}(i)\,\,\,\,\,\,\,\,\,\, & \,\,\,\,\,\,\,\,\,\,\,\,\,\,\,\,\,\,\,\, \bigoplus_i M^{2i+1}(i), \\ };
\path[->,font=\scriptsize]
(m-1-1) edge[out=30,in=150] node[above] {$\bigoplus_i d^{2i}$} (m-1-2)
(m-1-2) edge[out=210, in=330] node[below] {$\bigoplus_i d^{2i+1}$} (m-1-1);
\end{tikzpicture}
\end{center}
and that maps morphisms in the obvious way. It is clear that $F$ is an equivalence of abelian categories and preserves homotopies. For an acyclic complex $M^{\bullet}$, $F(M^{\cbullet})$ is a factorization which is the totalization of the acyclic complex of factorizations
$$ \ldots \xrightarrow{F(d_{-2})} F(M^{-1}) \xrightarrow{F(d_{-1})} F(M^0) \xrightarrow{F(d_{0})} F(M^1) \xrightarrow{F(d_{1})}\ldots . $$
Therefore $F$ induces an equivalence between     
$\dbcoh{X}$ and  $\dcoh{[X/\mathbb{G}_m], 0}$.
\end{proof}

\subsection{Semi-orthogonal decompositions} 
In this section we provide background material on semi-orthogonal decompositions and record a few facts we will need later.  Standard references are \cite{Bon, BK, BO95}.

\begin{definition}
 Let $\mathcal A \subseteq \mathcal T$ be a full triangulated subcategory. The \newterm{right orthogonal} $\mathcal A^\perp$ to $\mathcal A$ is the full subcategory of $\mathcal T$  consisting of objects $B$ such that $\op{Hom}_\cT(A,B) = 0$ for any $A \in \mathcal A$.
 The \newterm{left orthogonal} $\leftexp{\perp}{\mathcal A}$ is is the full subcategory of $\mathcal T$  consisting of objects $B$ such that $\op{Hom}_\cT(B,A) = 0$ for any $A \in \mathcal A$.
\end{definition}
The left and right orthogonals are naturally triangulated subcategories.

\begin{definition}\label{def:SO}
A \newterm{weak semi-orthogonal decomposition} of a triangulated category, $\mathcal T$, is a sequence of full triangulated subcategories, $\mathcal A_1, \dots ,\mathcal A_m$, in $\mathcal T$ such that $\mathcal A_i \subset \mathcal A_j^{\perp}$ for $i<j$ and, for every object $T \in \mathcal T$, there exists a diagram:
\begin{center}
\begin{tikzpicture}[description/.style={fill=white,inner sep=2pt}]
\matrix (m) [matrix of math nodes, row sep=1em, column sep=1.5em, text height=1.5ex, text depth=0.25ex]
{  0 & & T_{m-1} & \cdots & T_2 & & T_1 & & T   \\
   & & & & & & & &  \\
   & A_m & & & & A_2 & & A_1 & \\ };
\path[->,thick]
 (m-1-1) edge (m-1-3) 
 (m-1-3) edge (m-1-4)
 (m-1-4) edge (m-1-5)
 (m-1-5) edge (m-1-7)
 (m-1-7) edge (m-1-9)

 (m-1-9) edge (m-3-8)
 (m-1-7) edge (m-3-6)
 (m-1-3) edge (m-3-2)

 (m-3-8) edge node[sloped] {$ | $} (m-1-7)
 (m-3-6) edge node[sloped] {$ | $} (m-1-5)
 (m-3-2) edge node[sloped] {$ | $} (m-1-1)
;
\end{tikzpicture}
\end{center}
where all triangles are distinguished and $A_k \in \mathcal A_k$. We shall denote a weak semi-orthogonal decomposition by $\langle \mathcal A_1, \ldots, \mathcal A_m \rangle$. If $\mathcal A_i$ are essential images of fully-faithful functors, $\Upsilon_i: \mathcal A_i \to \mathcal T$, we may also denote the weak semi-orthogonal decomposition by
\begin{displaymath}
 \langle \Upsilon_1,\ldots,\Upsilon_m \rangle.
\end{displaymath}
\end{definition}

\begin{lemma} \label{lemma: functoriality of pieces of SOD}
 The assignments $T \mapsto T_i$ and $T \mapsto A_i$ appearing in the definition of a weak semi-orthogonal decomposition are unique and functorial. 
\end{lemma}

\begin{proof}
 This is standard, see e.g. \cite[Lemma 2.4]{Kuz09}. 
\end{proof}

Closely related to the notion of a semi-orthogonal decomposition is the notion of a left/right admissible subcategory of a triangulated category.

\begin{definition}
 Let $\alpha: \mathcal A \to \mathcal T$ be the inclusion of a full triangulated subcategory of $\mathcal T$. The subcategory, $\mathcal A$, is called \newterm{right admissible} if $\alpha$ has a right adjoint, denoted $\alpha^!$, and \newterm{left admissible} if $\alpha$ has a left adjoint, denoted $\alpha^*$. A full triangulated subcategory is called \newterm{admissible} if it is both right and left admissible.
\end{definition}

\begin{definition}
 A \newterm{semi-orthogonal decomposition} is a weak semi-orthogonal decomposition
 \[
 \langle \mathcal A_1, \ldots, \mathcal A_m \rangle
 \]
 such that each $\mathcal A_i$ is admissible.  The notation is left unchanged.
\end{definition}

\subsection{Elementary wall-crossings} \label{vgit} 

In this section, we review part of the relationship between variations of GIT quotients \cite{Tha96, DH98} and derived categories, following \cite{BFK12}. While consideration of the general theory was inspirational to our approach to Homological Projective Duality, it is sufficient for this paper to consider only the simplest types of variations of GIT quotients, namely elementary wall crossings.  

Let $Q$ be a smooth, quasi-projective variety and let $G$ be a reductive linear algebraic group. Let
\begin{displaymath}
 \sigma: G \times Q \to Q
\end{displaymath}
denote an action of $G$ on $Q$. Recall that a \newterm{one-parameter subgroup}, $\lambda: \mathbb{G}_m \to G$, is an injective homomorphism of algebraic groups.

From $\lambda$, we can construct some subvarieties of $Q$. We let $Z_{\lambda}^0$ be a choice of connected component of the fixed locus of $\lambda$ on $Q$. Set
\begin{displaymath}
 Z_{\lambda} := \{ q \in Q \mid \lim_{t \to 0} \sigma(\lambda(t),q) \in Z_{\lambda}^0\}.
\end{displaymath}
The subvariety $Z_{\lambda}$ is called the \newterm{contracting locus} associated to $\lambda$ and $Z_{\lambda}^0$. If $G$ is Abelian, $Z_{\lambda}^0$ and $Z_{\lambda}$ are both $G$-invariant subvarieties. Otherwise, we must consider the orbits
\[
S_{\lambda} := G \cdot Z_\lambda ,
S_{\lambda}^0 := G \cdot Z_\lambda^0.
\]
Also, let
\begin{displaymath}
 Q_{\lambda} := Q \setminus S_{\lambda}.
\end{displaymath}

We will be interested in the case where $S_{\lambda}$ is a smooth closed subvariety satisfying a certain condition.  To state this condition we need the following group attached to any one-parameter subgroup
\begin{displaymath}
 P(\lambda) := \{ g \in G \mid \lim_{\alpha \to 0} \lambda(\alpha) g \lambda(\alpha)^{-1} \text{ exists}\}.
\end{displaymath}

\begin{definition} \label{definition: HKKN strat}
 Assume $Q$ is a smooth variety with a $G$-action. An \newterm{elementary HKKN stratification} of $Q$ is a disjoint union
\[
\mathfrak{K}: Q = Q_{\lambda} \sqcup S_{\lambda},
\]
obtained from the choice of a one-parameter subgroup $\lambda: \mathbb{G}_m \to G$, together with the choice of a connected component, denoted $Z_{\lambda}^0$, of the fixed locus of $\lambda$ such that
 \begin{itemize}
   \item $S_{\lambda}$ is closed in $X$.
 \item The morphism,
\begin{align*}
 \tau_{\lambda}: [(G \times Z_{\lambda})/P(\lambda)] & \to S_{\lambda} \\
 (g,z) & \mapsto g \cdot z
\end{align*}
is an isomorphism where $p \in P(\lambda)$ acts by
\begin{displaymath}
 (p,(g,z)) \mapsto (gp^{-1},p \cdot z).
\end{displaymath}
 \end{itemize}
\end{definition}

We will need to attach an integer to an elementary HKKN stratification. We restrict the relative canonical bundle $\omega_{S_{\lambda}|Q}$ to any fixed point $q \in Z_{\lambda}^0$.  This yields a one-dimensional vector space which is equivariant with respect to the action of $\lambda$.  

\begin{definition}
 The \newterm{weight of the stratum} $S_{\lambda}$ is the $\lambda$-weight of $\omega_{S_{\lambda}/ Q}|_{Z_{\lambda}^0}$. It is denoted by $t(\mathfrak{K})$.
\end{definition}

Furthermore, given a one parameter subgroup $\lambda$ we may also consider its composition with inversion
\[
-\lambda(t) := \lambda(t^{-1}) = \lambda(t)^{-1},
\]
and ask whether this provides an HKKN stratification as well. This leads to the following definition.

\begin{definition}
 An \newterm{elementary wall-crossing}, $(\mathfrak{K}^+,\mathfrak{K}^-)$, is a pair of elementary HKKN stratifications,
 \begin{align*}
 Q = & Q_{\lambda} \sqcup S_{\lambda},\\
 Q = & Q_{-\lambda} \sqcup S_{-\lambda},
 \end{align*}
 such that $Z^0_\lambda = Z^0_{-\lambda}$. We often let $Q_+ := Q_{\lambda}$ and $Q_{-} := Q_{-\lambda}$.
\end{definition}

Let $C(\lambda)$ denote the centralizer of the $1$-parameter subgroup $\lambda$. For an elementary wall-crossing set
\[
\mu = - t(\mathfrak{K}^+) + t(\mathfrak{K}^-).
\]

\begin{theorem} \label{thm: bfkvgitlg}
 Let $Q$ be a smooth, quasi-projective variety equipped with the action of a reductive linear algebraic group, $G$. Let $w \in \op{H}^0(Q,\mathcal L)^G$ be a $G$-invariant section of a $G$-invertible sheaf, $\mathcal L$. Suppose we have an elementary wall-crossing, $(\mathfrak{K}^+,\mathfrak{K}^-)$, 
 \begin{align*}
  Q & = Q_+ \sqcup S_{\lambda} \\
  Q & = Q_- \sqcup S_{-\lambda},
 \end{align*}
and assume that $\mathcal L$ has weight zero on $Z_{\lambda}^0$ and that $S^0_{\lambda}$ admits a $G$ invariant affine open cover. Fix any $D\in \Z$. 
 
 \begin{enumerate}
  \item If $\mu > 0$, then there are fully-faithful functors,
  \begin{displaymath}
   \Phi^+_D: \dcoh{[Q_-/G],w|_{Q_-}} \to \dcoh{[Q_+/G],w|_{Q_+}},
  \end{displaymath}
  and, for $-t(\mathfrak{K}^-) + D \leq j \leq -t(\mathfrak{K}^+) + D -1$,
  \begin{displaymath}
   \Upsilon_j^+: \dcoh{[Z_{\lambda}^0/C(\lambda)],w_{\lambda}}_j \to \dcoh{[Q_+/G],w|_{Q_+}},
  \end{displaymath}
  and a semi-orthogonal decomposition,
  \begin{displaymath}
   \dcoh{[Q_+/G],w|_{Q_+}} = \langle  \Upsilon^+_{-t(\mathfrak{K}^-)+D}, \ldots, \Upsilon^+_{-t(\mathfrak{K}^+)+D-1}, \Phi^+_D  \rangle.
  \end{displaymath}
  \item If $\mu = 0$, then there is an exact equivalence,
  \begin{displaymath}
   \Phi^+_D: \dcoh{[Q_-/G],w|_{Q_-}} \to \dcoh{[Q_+/G],w|_{Q_+}}.
  \end{displaymath}
  \item If $\mu < 0$, then there are fully-faithful functors,
  \begin{displaymath}
   \Phi^-_D: \dcoh{[Q_+/G],w|_{Q_+}} \to \dcoh{[Q_-/G],w|_{Q_-}},
  \end{displaymath}
  and, for $-t(\mathfrak{K}^+) + D \leq j \leq -t(\mathfrak{K}^-) + D -1$,
  \begin{displaymath}
   \Upsilon_j^-: \dcoh{[Z_{\lambda}^0/C(\lambda)],w_{\lambda}}_j \to \dcoh{[Q_-/G],w|_{Q_-}},
  \end{displaymath}
  and a semi-orthogonal decomposition,
  \begin{displaymath}
   \dcoh{[Q_-/G],w|_{Q_-}} = \langle  \Upsilon^-_{-t(\mathfrak{K}^+)+D}, \ldots, \Upsilon^-_{-t(\mathfrak{K}^-)+D-1}, \Phi^-_D \rangle.
  \end{displaymath}
 \end{enumerate}
\end{theorem}

\begin{proof}
 This is \cite[Theorem 3.5.2]{BFK12}.
\end{proof}

The categories, $\dcoh{[Z_{\lambda}^0/C(\lambda)],w_{\lambda}}_j$, appearing in Theorem \ref{thm: bfkvgitlg} are the full subcategories consisting of objects of $\lambda$-weight $j$ in $\dcoh{[Z_{\lambda}^0/C(\lambda)],w_{\lambda}}$. For more details, we refer the reader to \cite{BFK12}. In our situation, we will only need the conclusion of the following lemma. We set 
\begin{displaymath}
 Y_{\lambda} : = [ Z_{\lambda}^0 / (C(\lambda)/\lambda) ].
\end{displaymath}

\begin{lemma} \label{lemma: wall compositions are same}
 We have an equivalence,
 \begin{displaymath}
  \dcoh{Y_{\lambda},w_{\lambda}} \cong \dcoh{[Z_{\lambda}^0/C(\lambda)],w_{\lambda}}_0.
 \end{displaymath}

 Further, assume that there there is a character, $\chi: C(\lambda) \to \mathbb{G}_m$, such that
 \begin{displaymath}
  \chi \circ \lambda(t) = t^l.
 \end{displaymath}
 Then, twisting by $\chi$ provides an equivalence,
 \begin{displaymath}
  \dcoh{[Z^0_{\lambda}/C(\lambda)],w_{\lambda}}_r  \cong \dcoh{[Z^0_{\lambda}/C(\lambda)],w_{\lambda}}_{r+l},
 \end{displaymath}
 for any $r \in \Z$.
\end{lemma}

\begin{proof}
 This is Lemma 3.4.4 of \cite{BFK12}; we give the very simple and short proof here. A quasi-coherent sheaf on $Y_{\lambda}$ is a quasi-coherent $C(\lambda)$-equivariant sheaf on $Z_{\lambda}^0$ for which $\lambda$ acts trivially, i.e. of $\lambda$-weight zero.  For the latter statement just observe that twisting with $\chi$ is an autoequivalence of $\dbcoh{[Z_{\lambda}^0/C(\lambda)]}$ which brings range to target and its inverse does the reverse.
\end{proof}

\subsection{Homological Projective Duality}

In this section, we provide an introduction to Homological Projective Duality (HPD) following \cite{KuzHPD}. To make the ideas more transparent, we start by considering HPD over a point. We then make the definitions we need for the relative setting considered in the rest of this paper. 

Let $X$ be a smooth projective variety equipped with a morphism $f: X \to \P(V)$. 

We have a canonical section of $\mathcal O_{\P(V)}(1) \boxtimes \mathcal O_{\P(V^*)}(1)$ determined as follows. Under the natural isomorphism,
\begin{displaymath}
 V \otimes V^* \cong \op{End}(V),
\end{displaymath}
the identity map on $V$ corresponds to an element $u \in V \otimes V^*$. We define a section
\begin{displaymath}
 \theta_V \in \Gamma(\mathcal O_{\P(V)}(1) \boxtimes \mathcal O_{\P(V^*)}(1)) \cong V\otimes V^* 
\end{displaymath}
by taking the image of $u$ under the isomorphism above. 

Let $\mathcal O_X(1)$ denote the pullback $f^*\mathcal O_{\P(V)}(1)$. We can also pull back $\mathcal O_{\P(V)}(1) \boxtimes \mathcal O_{\P(V^*)}(1)$ to $X \times \P(V^*)$. Let $\theta_X$ denote the pull back of $\theta_V$.

\begin{definition}
 The zero locus of $\theta_X$ is called the \newterm{universal hyperplane section} of $f$. It is denoted by $\cX$.
\end{definition}

The universal hyperplane section comes equipped with two natural morphisms, 
\begin{displaymath}
p: \cX \to X \tand q: \cX \to \P(V^*).
\end{displaymath}
The fiber of $q$, $\mathcal X_H$, over $H \in \P(V^*)$ is exactly the hyperplane section of $X$ corresponding to $H$. 

\begin{remark}
 Recall that when $X$ is smooth and $f$ is an embedding, the projective dual to $X$ is the closed subset
 \begin{displaymath}
  X^{\vee} := \{H \in \P(V^*) \mid X_H \text{ is singular}\}.
 \end{displaymath}
 with its reduced, induced scheme structure. Thus $X^{\vee}$ is the non-regular, i.e. critical, locus of $q: \cX \to \P(V^*)$ in $\P(V^*)$.
\end{remark}

Homological Projective Duality is a phenomenon that can be considered as a lifting of the notion of classical projective duality to non-commutative geometry. 
The starting data for HPD is a smooth variety, $X$, together with a map to a projective space, $f: X \to \P(V)$, and a special type of a semi-orthogonal decomposition called a Lefschetz decomposition.  We now provide the setup to define a Lefschetz decomposition.

\begin{definition}
  Let $B$ be an algebraic variety and $\mathcal T$ triangulated category.  A $B$\newterm{-linear structure} on $\mathcal T$ is a $\dperf(B)$-module structure
\[
 F: \dperf(B) \otimes \cT \to \cT
\]
on $\cT$.  
\end{definition}

\begin{definition}
An exact functor
\[
\Phi : \mathcal T \to \mathcal T'
\]
between $B$-linear triangulated categories with respect to $F$ and $F'$ is called $B$\newterm{-linear} if there are bi-functorial isomorphisms 
\[
\Phi(F(A \otimes T)) \cong F'(\Phi(A) \otimes T)
\]
for any $T \in \mathcal T, A \in \dperf(B)$.
\end{definition}

Now let $B = \P(V)$ and consider a $\P(V)$-linear category $\mathcal T$ with respect to $F$.  To simplify notation, denote by $(s)$ the functor of tensoring with $F(\O(s))$.  

\begin{definition}
A \newterm{Lefschetz decomposition} of a $\P(V)$-linear category $\mathcal T$ is a semi-orthogonal decomposition of the form,
\[
\mathcal T =\langle \mathcal A_0,\mathcal A_1(1),\ldots,\cA_i(i)\rangle,
\]
where 
\[
0\subset\mathcal A_i\subset\mathcal A_{i-1}\subset\ldots\subset\mathcal A_1\subset\mathcal A_0\subset\mathcal T
\]
is a chain of admissible subcategories of $\mathcal T$ and $\mathcal A_s(s)$ denotes the essential image of the category $\mathcal A_s$ after application of the functor $(s)$. 
\end{definition}

\begin{definition}
A \newterm{dual Lefschetz decomposition} of a $\P(V^*)$-linear category $\mathcal T'$ is a semi-orthogonal decomposition of the form,
\[
\mathcal T' =\langle \mathcal B_{j-1}(1-j),\mathcal B_{j-2}(2-j),\ldots,\mathcal B_0 \rangle ,
\]
where 
\[
0\subset\mathcal B_{j-1}\subset\mathcal B_{j-2}\subset\ldots\subset\mathcal B_1\subset\mathcal B_0\subset\mathcal T'
\]
is a chain of admissible subcategories of $\mathcal T'$ and $\mathcal B_s(s)$ denotes the essential image of the category $\mathcal B_s$ after application of the functor $(s)$. 
\end{definition}

 Now consider a morphism $f: X \to \P(V)$.  The most important property of a Lefschetz decomposition, given by the following proposition, is that it induces a semi-orthogonal decomposition on the derived category of any linear section of $X$.  This result, and the proposition succeeding it follow from the results in \cite{KuzHPD} and \cite{KuzBaseChange}. We give proofs for these two statements for the sake of completion.  

\begin{proposition} \label{prop: X breakdown}
Consider a morphism $f: X \to \P(V)$ and a Lefschetz decomposition
\[
\dbcoh{X} = \langle \mathcal A_0,\mathcal A_1(1),\ldots,\cA_i(i)\rangle
\]
with respect to $f^*$.
Let $L \subseteq V^*$ be a linear subspace of dimension $r$, $L^{\perp}$ its orthogonal in $V$, and let
\[
X_L := X \times_{\P(V)} \P(L^{\perp})
\]
be a complete linear section of $X$ i.e. $\op{dim} \ X_L = \op{dim} \ X - \op{dim} \ L$. There is a semi-orthogonal decomposition
 \begin{displaymath}
  \dbcoh{X_L}= \langle \mathcal C_L,\mathcal A_r(r),\ldots,\cA_i(i) \rangle.
 \end{displaymath}
 of $\dbcoh{X_L}$ where the functor,
 \begin{displaymath}
  \mathcal A_j(j) \to \dbcoh{X_L},
 \end{displaymath}
 is the composition, 
 \begin{displaymath}
  \mathcal A_j(j) \ra \dbcoh{X} \ra \dbcoh{X_L}
 \end{displaymath}
 of the inclusion and derived restriction to $X_L$.
\end{proposition}

\begin{proof}
Let $\delta: X_L \to X$ be the inclusion. 
Let $A_s\in \cA_s(s)$ and $A_t \in \cA_t(t)$.  
Restrict the Koszul resolution on $L$ to $X$ to obtain an exact complex 
\[
0 \to \wedge^r L \otimes_k \O_X(-r) \to \cdots \to   \O_X \to \O_{X_L} \to 0
\]
and tensor this complex with $A_t \otimes_{\O_X} A_s^{\vee}$ to get
\[
  0 \to \wedge^r L \otimes_k A_s \otimes_{\O_X} A_t^{\vee}(-r) \to \cdots \to A_s \otimes_{\O_X} A_t^{\vee} \to \delta^*A_s \otimes_{\O_{X_L}} \delta^*A_t^{\vee} \to 0.
\]
Applying global sections yields an exact sequence of hypercohomology
\begin{equation} \label{eq: hypercohomology}
0 \to \wedge^r L \otimes_k \mathbf{R} \! \op{Hom}_{X}(A_t, A_s(-r)) \to \cdots \to \mathbf{R}\!\op{Hom}_{X}(A_t, A_s) \to \mathbf{R}\!\op{Hom}_{X_L}(\delta^*A_t, \delta^*A_s) \to 0.
\end{equation}
Now, by definition of a Lefschetz decomposition 
$
\mathbf{R}\!\op{Hom}_{X_L}(A_t, A_s(-p)) =0
$
if $p \leq s < t$.  Plugging into \eqref{eq: hypercohomology} we obtain
\[
\mathbf{R} \! \op{Hom}_{X_L}(\delta^*A_t, \delta^*A_s)  \cong 
\begin{cases}
\mathbf{R}\!\op{Hom}_X(A_t, A_s) & \tif r \leq s = t \\
0 & \tif  r \leq s < t,
\end{cases}
\]
which shows that $\delta^*$ is fully faithful on $\cA_t(t)$ for $t\geq r$ and that the images of these subcategories are semi-orthogonal.
\end{proof}
 
A Lefschetz decomposition also induces a semi-orthogonal decomposition on the universal hyperplane section $\cX$ with respect to $f$ and similarly on the family of hyperplane sections over any $L \subseteq V^*$, 
\[
\cX_L := \cX \times_{\P(V^*)} \P(L).
\]
Let $\pi_L$ denote the natural map from $\mathcal X_L$ to $\P(L)$ and define $\mathcal A_k(k)\boxtimes \dbcoh{\P(L})$ to be the full triangulated subcategory of $\dbcoh{X\times \P(L)}$ generated by objects $\mathcal F\boxtimes\mathcal G$, with $\mathcal F\in A_k(k)\subset\dbcoh{X}$ and $\mathcal G\in\dbcoh{\P(L)}$. 

\begin{proposition} \label{prop: universal breakdown}
For any  Lefschetz decomposition,
 \begin{displaymath}
  \dbcoh{X} = \langle \mathcal A_0, \mathcal A_1(1),\ldots\cA_i(i) \rangle,
 \end{displaymath}
 of $\dbcoh{X}$ there is an associated semi-orthogonal decomposition,
 \begin{equation} \label{eq: decomp}
  \dbcoh{\mathcal X_L}= \langle \mathcal D_L, \mathcal A_1(1)\boxtimes \dbcoh{\P(L)},\ldots,\cA_i(i)\boxtimes \dbcoh{\P(L)}\rangle
\end{equation}
 where $\mathcal D_L$ is defined as the right orthogonal to $ \langle \mathcal A_1(1)\boxtimes \dbcoh{\P(L)},\ldots,\cA_i(i)\boxtimes \dbcoh{\P(L)} \rangle$.
\end{proposition}

\begin{proof}
Notice that we get a semi-orthogonal decomposition
\[
\dbcoh{X \times \P(L)} = \langle \cA_0 \boxtimes \dbcoh{\P(L)}, \ldots, \cA_i(i) \boxtimes \dbcoh{\P(L)} \rangle.
\]
Now, consider $X \times \P(L)$ with the Segre embedding and apply Proposition~\ref{prop: X breakdown} to get the result. \end{proof}

The following is Definition 6.1 of \cite{KuzHPD}.
\begin{definition} \label{def: original HPD}
  Given $f: X \rightarrow \P(V)$ and a Lefschetz decomposition $\langle \cA_0 , \cA_1(1) ,\dots \cA_i(i)\rangle$ of $\dbcoh{X}$, a \newterm{Homological Projective Dual} $Y$ is an algebraic variety  together with a morphism $g: Y \rightarrow \P(V^*)$  and a fully-faithful Fourier-Mukai transform $\Phi_{\mathcal P}$ with kernel $\mathcal P \in \dbcoh{Y \times_{\P(V^*)}\cX}$ which induces a semi-orthogonal decomposition
   \[
     \dbcoh{\mathcal X}=\langle \Phi_{\mathcal P}(\dbcoh{Y}) ,\mathcal A_1(1)\boxtimes \dbcoh{\P(V^*)},\ldots,\cA_i(i)\boxtimes \dbcoh{\P(V^*)}\rangle.
  \]
\end{definition}

The Fundamental Theorem of HPD relates linear sections in $X$ with respect to $f$ to their dual linear sections of $Y$ with respect to $g$. Let $N$ be the dimension of $V$ and  $L\subset V^*$ be a linear subspace of dimension $r$.  Recall that
\[
X_L:=X\times_{\P(V)}\P(L^\perp)
\]
and define
\[
Y_L:=Y\times_{\P(V^*)}\P(L).
\]

\begin{theorem}[Fundamental Theorem of Homological Projective Duality] \label{thm: fundamental HPD}
 Let $Y\rightarrow \P(V^*)$ be a homological projective dual to $X\rightarrow \P(V)$  with respect to the Lefschetz decomposition $\left\{ \cA_i \right\}$ in the sense of Definition~\ref{def: original HPD}. With the notation above we have the following:
\begin{itemize}

\item The category $\dbcoh{Y}$ admits a dual Lefschetz decomposition  
 $$\dbcoh{Y}=\langle\mathcal B_j(-j),\ldots,\mathcal B_1(-1),\mathcal B_0\rangle$$.
\item Assume that $X_L$ and $Y_L$ are complete linear sections, i.e.
\[
\op{dim}(X_L) = \op{dim}(X)-r \tand \op{dim}(Y_L) = \op{dim}(Y)+r-N.
\]
Then there exist semi-orthogonal decompositions,
\[
\dbcoh{X_L}=\langle\mathcal C_L,\mathcal A_r(1),\ldots,\mathcal A_i(i-r+1)\rangle,
\]
 and,
 \[
 \dbcoh{Y_L}=\langle\mathcal B_j(N-r-j-1),\ldots,\mathcal B_{N-r}(-1),\mathcal C_L\rangle.
 \]
\end{itemize}
\end{theorem}

\begin{proof}
 This is \cite[Theorem 6.3]{KuzHPD}.
\end{proof}

\begin{remark} \label{rem: HPD table}
Figure~\ref{fig: HPD Table} is a useful representation of the pieces appearing in the semi-orthogonal decompositions in the theorem above.   
The boxes themselves represent what Kuznetsov calls primitive subcategories $\mathfrak a_s := \cA_s / \cA_{s+1}$.
The longer vertical line is placed at $r$, the dimension of $L$.  The shaded boxes to the right of the long vertical line represent the terms of the perpendicular to $\mathcal C_L$ in $\dbcoh{X_L}$.
The shaded boxes to the left of the vertical line represent the terms of the perpendicular to $\mathcal C_L$ in the  derived category of the homological projective dual $Y_L$.  In the $i^{\op{th}}$ column, the category generated by the boxes below the staircase correspond to $\cA_{i-1}$ and the category generated by the boxes above the staircase give $\cB_{j-i+1}$.
\end{remark}

\newcommand*{\xMin}{1}
\newcommand*{\xMax}{12}
\newcommand*{\xShift}{5}
\newcommand*{\yMin}{0}
\newcommand*{\yMax}{6}
\newcommand*{\yMinshift}{1}
\newcommand*{\yMaxshift}{5}
\newcommand*{\yMaxx}{5}
\newcommand*{\yShift}{3}
\begin{center}
\begin{figure}
\tiny
\begin{tikzpicture}

\filldraw[fill=black!20!white,draw=white!100]
  (2,6) -- (4, 6)  -- (4,4) -- (3,4) -- (3,5) -- (2,5) -- (2,6);

\filldraw[fill=black!20!white,draw=white!100]
(4,3) -- (5,3) -- (5,2) -- (6,2) -- (6, 1) -- (7,1) -- (7,0) -- (4,0) -- (4,3);
    \foreach \i in {\xMin,...,\xMax} {
        \draw [thin,black] (\i,\yMin) -- (\i,\yMax)  node [below] at (\i,\yMin) {};
    }
    \foreach \i in {\yMinshift,...,\yMaxshift} {
        \draw [thin,black,dashed] (\xMin,\i) -- (\xMax,\i) node [left] at (\xMin,\i) {};
    }
        \node at (-.2, \yMin +.5) {$\mathfrak a_s$};
            \node at (-.2, \yMax -.5) {$\mathfrak a_0$};

    \foreach \i in {\yMin,\yMax} {
        \draw [thin,black] (\xMin,\i) -- (\xMax,\i) node [left] at (\xMin,\i) {};
    }
\node at (\xMin+.5, \yMin - .5) {$\cA_0$};
\node at (\xMin+.5 + \xShift, \yMin - .5) {$\cA_{i}$};

\node at (\xMax - .5, \yMax + .5) {$\cB_0$};
\node at (\xMin+1.5, \yMax +.5) {$\cB_{j}$};

\draw[<->, blue, very thick] (4, \yMin-.5) -- (4, \yMax+.5);

    \foreach \i in {\yMin,...,\yMaxx} {
        \draw [very thick, red] (\xMin+\i,\yMax- \i) -- (\xMin+\i +1,\yMax- \i); 
    }
        \foreach \i in {\yMin ,...,\yMaxx} {
        \draw [very thick, red] (\xMin+\i +1,\yMax- \i) -- (\xMin+\i +1,\yMax- \i -1); 
    }

\draw [very thick, red] (\xMin + \yMax, \yMin) -- (\xMax, \yMin);

\end{tikzpicture}
\normalsize
\captionof{figure}{Kuznetsov's image of Lefschetz collections and their duals} \label{fig: HPD Table}
\end{figure} 
\end{center}

\begin{remark}
Homological Projective Duality is a duality in the following sense. If $Y\ra\P(V^*)$ is a homological projective dual to $X\ra\P(V)$, then $Y\ra\P(V^*)$ has a dual Lefschetz decomposition $\dbcoh{Y}=\langle\mathcal B_j(-j),\ldots,\mathcal B_1(-1),\mathcal B_0\rangle$. By dualizing as in Theorem 7.3 in {\em loc. cit.}, we get a Lefschetz decomposition $\dbcoh{Y}=\langle\mathcal B_0^*,\ldots,\mathcal B_j^*(j)\rangle$ for $Y\ra \P(V^*)$. With respect to this Lefschetz decomposition, Kuznetsov shows that $X\ra\P(V)$ is a homological projective dual to $Y\ra\P(V^*)$.
\end{remark}

Let $X$ be an $S$-scheme and $\cE$ be a locally-free coherent sheaf over $S$. Let $f:X\ra \P_{S}(\cE)$ be an $S$-morphism. We now consider Homological Projective Duality in the relative setting. This was already studied by Kuznetsov when $\cE$ is the trivial bundle \cite[Theorem 6.27]{KuzHPD} and, in the case of relative $2$-Veronese embeddings, by Auel, Bernardara, and Bolognesi \cite[Theorem 1.13]{ABB}.

The definition of a Lefschetz decomposition extends to the relative setting by replacing the projective space $\P(V)$ by the projectivization $\P_S(\cE)$ and $\cO_{\P(V)}(1)$ by $\cO_{\P_S(\cE)}(1)$. In the relative setting, we define the universal hyperplane section as $\cX = X \times_{\P_S(\cE)} \cX_0 $ where $\cX_0$ is the incidence variety in $\P_S(\cE)\times_S \P_S(\cE^*)$.

\begin{definition}\label{defofhpd}
 Given an $S$-morphism $f: X \rightarrow \P_S(\cE)$ and a Lefschetz decomposition $\langle \cA_0 , \cA_1(1) ,\dots \cA_i(i)\rangle$, a \newterm{weak Homological Projective Dual} $Y$ relative to $S$ is either
 \begin{itemize}
  \item an $S$-scheme $Y$ together with a morphism $Y \rightarrow \P_S(\cE^*)$, or
  \item a  gauged Landau-Ginzburg model $(Q, G, \mathcal L, w)$ (Definition \ref{deflgmodel}) together with an $S$-morphism $g: Q \to \P_S(\cE^*)$,
 \end{itemize}
 such that there is a  semi-orthogonal decomposition
   \[
     \dbcoh{\mathcal X}=\langle \Phi ,\mathcal A_1(1)\boxtimes \dbcoh{\P_S(\cE^*)},\ldots,\cA_i(i)\boxtimes \dbcoh{\P_S(\cE^*)}\rangle,
  \]
  where $\Phi$ denotes the essential image of a fully-faithful $\P_S(\cE)$-linear functor
  \begin{displaymath}
    \Phi: \dbcoh{Y} \to \dbcoh{\mathcal X}
  \end{displaymath}
  or 
  \begin{displaymath}
   \Phi: \dcoh{[Q/G],w} \to \dbcoh{\mathcal X},
  \end{displaymath}
\end{definition}
with the $\P_S(\cE)$-linear structure given by tensoring with pullbacks of objects in $\dcoh{\P_S(\cE)}$.

\begin{remark}
  The difference, between Kuznetsov's definition of Homological Projective Dual and the above definition of a weak Homological Projective Dual is in the assumption that the functor $\Phi$ is given by a Fourier-Mukai kernel in the fiber product $\dbcoh{Y\times_{\P(V^*)}\cX}$. Recent work by Ben-Zvi, Nadler and Preygel \cite{BZNP} shows that a Fourier-Mukai kernel in $\dbcoh{Y\times_{\P(V^*)}\cX}$ for $\Phi$ does exist when $Y$ is a scheme and $\P(V^*)$-linearity is interpreted in a stronger, $\infty$-categorical sense.  
Because of this difference, we prove the conclusions of the Fundamental Theorem of HPD separately in the setup that we consider in this paper (Theorem \ref{thm: MainHPDMF}).

\end{remark}

\begin{remark}
  In the relative setting, we will consider, instead of linear sections $X_L$ and $Y_L$, the fiber products $X \times_{\P_S(\cE)} \P_S(\cW)$ and    $Y \times_{\P_S(\cE^*)} \P_S(\cV)$ where $\cV = \cE^*/\cU$ is a quotient bundle and $\cW = \cE/\cU^\perp$. For a Landau-Ginzburg pair $(Q,G,\cL,w)$, we define the fiber product as $(Q,G,\cL,w)\times_{\P_S(\cE^*)} \P_S(\cV):=(Q\times_{\P_S(\cE^*)} \P_S(\cV),G,\cL_{|\P_S(\cV)},w_{|\P_S(\cV)})$.

\end{remark}

\section{Homological Projective Duality and VGIT}\label{sec:lghpd}
In this section we construct a weak homological projective dual to a GIT quotient provided we are also given the data of an elementary wall-crossing.  

The idea behind this section is to start with a variety $X$ given as a quotient $X = [Q^{\op{ss}}(\cM) / G] = [Q_+ / G] $, where $Q$ is a smooth variety with an action of $G$ and $\cM$ is a linearization of the $G$-action, and to then prove that, under basic assumptions, an elementary wall crossing which varies the GIT quotient induces a Lefschetz decomposition of $\dbcoh{X}$, with respect to the morphism $X\to \P_S(\cE)$ induced by the bundle $\cM$. Moreover, the same data can be used to construct a weak homological projective dual to $X$ which is a Landau-Ginzburg pair $(Y,w)$, where $Y$ is a GIT quotient of the space 
$\op{V}_Q(\cM) \times_S \op{V}_S(\cE^*)$, and $w$ is induced by the canonical section of $\mathcal O_{\P_S(\cE)\times_S\P_S(\cE^*)}(1,1)$. 

We follow the notation of Section \ref{vgit}.

\subsection{Lefschetz decompositions and HPD from elementary wall crossings}
Let $Q$ be a smooth quasi-projective variety equipped with the action of a reductive linear algebraic group $G$ and a morphism
\[
p: [Q/G] \to S.
\]

Let $\lambda$ be a one-parameter subgroup of $G$ which determines an elementary wall-crossing $(\mathfrak{K^+}, \mathfrak{K^-})$
 \begin{align*}
  Q & = Q_+ \sqcup S_{\lambda} \\
  Q & = Q_- \sqcup S_{-\lambda},
 \end{align*}
such that $S_{\lambda}^0=G\cdot Z_\lambda^0$ admits a $G$-equivariant affine cover and $S_\lambda$ has codimension at least $2$.   We let
\[
 \mu = - t(\mathfrak{K^+})+ t(\mathfrak{K^-})
\]
and  we assume that $\mu\geq 0$.

Assume that $G$ acts freely on $Q_+$ and that $X:=[Q_+/G]$ is a smooth and proper variety.  Notice that $X$ is an $S$-scheme by composing the inclusion with $p$.  We denote this map by
\[
g: X \to S.
\] 

Let $\mathcal E$ be a locally-free coherent sheaf of rank $N$ over $S$.  One can consider the projective bundle
\[
\P_S(\mathcal E) := [ ( \op{V}_S(\mathcal E) \backslash \mathbf{0}_{\op{V}_S(\mathcal E)}) / \gm]
\]
where $ \mathbf{0}_{\op{V}_S(\mathcal E)}$ denotes the zero-section of $\op{V}_S(\mathcal E)$.  This bundle comes with a projection 
\[
\pi : \P_S(\mathcal E) \to S.
\]
We denote the relative bundle $\cO_\pi(1)$ by $\O_{\P_S(\cE)}(1)$. Note that with our notation, $\pi_*\O_{\P_S(\cE)}(1)\cong\cE$.

Consider an $S$- morphism
\[
f: X \to \P_S(\mathcal E).
\]
We write 
\[
\cL := f^*\O_{\P_S(\mathcal E)} (1).
\]
Now suppose that there exists a $G$-equivariant invertible sheaf $\cM = \cO(\chi)$ on $Q$, for some character $\chi$ of $G$ such that, as an invertible sheaf on $[Q/G]$, it restricts to $\cL$ on $[Q_+/G]$
\[
\cM|_{[Q_+/G]} \cong \cL.
\]
Furthermore, let $d$ be the $\lambda$-weight of $\cM$. Recall that in this case, the $\lambda$-weight of $\cM = \cO(\chi)$ is the integer $d$, such that $\chi\circ\lambda(t)=t^d$. We assume that $d>0$.

Recall that we have fully-faithful functors
\begin{displaymath}
  \Upsilon_j^+: \dbcoh{[Z_{\lambda}^0/C(\lambda)]}_j \to \dbcoh{X}
\end{displaymath}
from by applying Theorem \ref{thm: bfkvgitlg} with $w=0$, and $\gm$ acting trivially, and using Proposition \ref{cor: isik}. Therefore, when writing semi-orthogonal decompositions, we will denote the essential images of the functors $\Upsilon_j^+$ by $\cZ_j^+$. By Lemma \ref{lemma: wall compositions are same}, we see that 
$$\dbcoh{[Z_{\lambda}^0/C(\lambda)]}_0\cong\dcoh{Y_{\lambda}},$$
where $Y_{\lambda} : = [ Z_{\lambda}^0 / (C(\lambda)/\lambda) ]$ and twisting by $\chi_{|C(\lambda)}$, which by definition is tensoring with the restriction of $\cL$, induces an isomorphism between $\cZ_n^+$ and that of $\cZ_{n+d}^+$ for any $n \in \Z$.

When $\mu \geq 0$, the elementary wall crossing induces a semi-orthogonal decomposition on $\dbcoh{X}$, which is a Lefschetz decomposition when $X$ is considered together with the map $f$ to $\P_S(\cE)$. The fineness of the Lefschetz decomposition depends on the $\lambda$-weight $d$ of $\cM$.

\begin{proposition} \label{prop: Lef dec from VGIT}
 If $\mu \geq 0$, there is a Lefschetz decomposition 
 $$\dbcoh{X}=\langle\cA_0,\ldots,\cA_i(i)\rangle$$of $X$ with respect to $f$, where $i=\lceil\frac{\mu}{d} \rceil -1$ and
\[
  \cA_j = \begin{cases}
  \langle \dbcoh{[Q_- / G]}, \cZ_0^+, \ldots, \cZ_{d-1}^+ \rangle & j=0 \\
  \langle \cZ_0^+, \ldots, \cZ_{d-1}^+ \rangle & 0 < j < \lceil\frac{\mu}{d} \rceil -1  \\
  \langle \cZ_0^+, \dots,\cZ_{\mu-d(\lceil\frac{\mu}{d}\rceil-1)}^+ \rangle & j = \lceil\frac{\mu}{d} \rceil -1 \\
    \end{cases}
  \]
\end{proposition}

\begin{proof}
 Taking $D=t(\mathfrak{K}^-)$ in Theorem \ref{thm: bfkvgitlg}, in combination with Proposition \ref{cor: isik}, gives a fully-faithful functor, $\Phi^+_{-t(\mathfrak{K}^-)}: \dbcoh{[Q_- / G]} \to \dbcoh{X}$, and a weak semi-orthogonal decomposition
 \begin{displaymath}
  \dbcoh{X} = \langle  \cZ^+_{0}, \ldots, \cZ^+_{\mu-1}, \cD  \rangle,
 \end{displaymath}
 where $\cD$ represents the essential image of the functor $\Phi^+_{t(\mathfrak{K}^-)}$.

 Since $X$ is smooth and proper, $\dbcoh{X}$ is saturated \cite[Corollary 3.1.5]{BvdB} and so is any weak semi-orthogonal component. By \cite[Proposition 2.8]{BK}, all the subcategories are fully admissible and we can mutate to get a new semi-orthogonal decomposition
 \begin{displaymath}
  \dbcoh{X} = \langle\cD, \cZ^+_{0}, \ldots, \cZ^+_{\mu-1}  \rangle.
 \end{displaymath}
 We conclude the proof by noticing, as above, that tensoring by $\cL$ induces an isomorphism between $\cZ_n^+$ and that of $\cZ_{n+d}^+$ for any $n \in \Z$.
\end{proof}

Recall that $\cX_0$ be the incidence scheme in $\P_S(\cE)\times_S \P_S(\cE^*)$ and that
\begin{displaymath}
 \cX = X \times_{\P_S(\cE)} \cX_0
\end{displaymath}
is the relative universal hyperplane section of the $S$-morphism $f: X \to \P_S(\cE)$.

We will now set up an elementary wall crossing for an action of $\widetilde{G} = G\times \GG_m \times \GG_m$ on a space $U_{\cE^*}^1$ with a potential function $w$ such that
\[
\dbcoh{\cX} \cong \dcoh{[(U_{\cE^*}^1)_+/ \widetilde{G}], w}.
\]
The gauged Landau-Ginzburg model corresponding to the quotient $[(U_{\cE^*}^1)_-/ \widetilde{G}]$ obtained from the elementary wall crossing will be our weak homological projective dual. 

Let us define
\begin{equation*}
  U_{\cE^*}^1 = \op{V}_Q(\cM) \times_S (\op{V}_S(\cE^*)\backslash \mathbf{0}_{\op{V}_S(\cE^*)} )
\end{equation*}
with an action of $\widetilde{G} = G\times \GG_m \times \GG_m$ which can be described by 
\begin{center}
\begin{tabular}{c c c}
   $\,$  & $\op{V}_Q(\cM)$ & $(\op{V}_S(\cE^*)\backslash \mathbf{0}_{\op{V}_S(\cE^*)} )$  \\
   $G$ &   $g$    &  $1$ \\ 
   $\GG_m$  & $\alpha_1^{-1}$  & $\alpha_1$ \\
   $\GG_m$  & $\alpha_2$ &  $1$
\end{tabular}
\end{center}
Here, $\alpha_1\in\GG_m$ and $\alpha_2\in\GG_m$ act by dilation on the fibers of the two respective bundles and the action of $G$ on $\op{V}_Q(\cM)$ is induced by the equivariant structure of $\cM$.

Let $\lambda_1$ be the one-parameter subgroup given by $\lambda_1(\alpha) = (\lambda(\alpha),1,1)$. The contracting locus for $\lambda_1$ is $$S_{\lambda_1} = \op{V}_{S_\lambda}(\cM|_{{S}_{\lambda}}) \times_S (\op{V}_S(\cE^*) \backslash \mathbf{0}_{\op{V}_S(\cE^*)}),$$ while the contracting locus for $-\lambda_1$ is
$$
  S_{-\lambda_1} = \mathbf{0}_{\op{V}_{S_{-\lambda}}(\cM)}\times_S (\op{V}_S(\cE^*)\backslash \mathbf{0}_{\op{V}_S(\cE^*)} ).
$$
Therefore, \begin{eqnarray*}
  (U_{\cE^*}^1)_+ = \op{V}_{Q_+}(\cM) \times_S (\op{V}_S(\cE^*)\backslash \mathbf{0}_{\op{V}_S(\cE^*)} )   \subset U_{\cE^*}^1.
\end{eqnarray*} and, by definition,
\begin{eqnarray*}
  (U_{\cE^*}^1)_- = U_{\cE^*}^1 \backslash S_{-\lambda_1}.
\end{eqnarray*} We will prove further below that $\lambda_1$ determines an elementary wall-crossing. 

We now show that the pull-back of the natural pairing $$\theta_{\cE} \in \Gamma( \P_S(\cE) \times_S \P_S(\cE^*) , \mathcal O_{\P_S(\cE)\times_S\P_S(\cE^*)}(1,1))$$ to $X \times_S \P(\cE^*)$, i.e. the section whose zero-scheme is $\cX$, induces a $G\times\GG_m$-invariant function $w$ on $U_{\cE^*}^1$, where $G\times \gm$ is $G\times\gm \times \left\{ 1 \right\} \subset \widetilde{G}$. 
Indeed, we have isomorphisms
\begin{eqnarray*}
  \Gamma(X\times _S \P_S(\cE^*), \op{Sym} (\cL \boxtimes \cO_{\P_S(\cE^*)}(1))) & \iso & \Gamma(\op{V}_{X\times_S \P_S(\cE^*)}(\cL \boxtimes \cO_{\P_S(\cE^*)}(1)),\cO) \\
  & \iso &  \Gamma( [\op{V}_{Q_{+}}(\cM) \times_{S} (\op{V}_{S}(\cE^*) \backslash \mathbf{0}_{\op{V}_{S}(\cE^{*})}) / (G\times \gm) ],\cO) \\
  & = &  \Gamma([(U_{\cE^*}^1)_+/(G\times \GG_m)],\cO) \\
  & \iso & \Gamma([U_{\cE^*}^1/(G\times \GG_m)],\cO) 
\end{eqnarray*}
where the first line comes from the fact that $\mathcal O_{\P_S(\cE)\times_S\P_S(\cE^*)}(1,1)|_{X \times_S \P_S(\cE^*)} = \mathcal L \boxtimes \mathcal O_{\P_S(\cE^*)}(1)$, and the last line comes from our assumption that $S_\lambda$ had codimension at least two in $Q$, which implies that $S_{\lambda_1}$ has codimension at least two in $U_{\cE^*}^1$.
Furthermore, $w$ is homogeneous of degree $1$ with respect to the final $\mathbb{G}_m$ component, i.e. it is semi-invariant with respect to the action of all of $\widetilde{G} = G\times \GG_m \times \GG_m$, with character $\beta(g,\alpha_1,\alpha_2) = \alpha_2$ of $\widetilde{G}$.  

We are now ready to state:

\begin{theorem}\label{thm: HPDMF}
 Let $Q$ be a smooth quasi-projective variety equipped with the action of a reductive algebraic group $G$ and a morphism $p: [Q/G] \to S$. Let $\lambda$ be a one-parameter subgroup of $G$ which determines an elementary wall-crossing $(\mathfrak{K^+}, \mathfrak{K^-})$ such that $S_{\lambda}^0$ admits a $G$-equivariant affine cover and $S_{\lambda}$ has codimension at least $2$ in $Q$. Assume that $X = [Q_+/G]$ is a smooth and proper variety and let $f: X \to \P_S(\cE)$ be an $S$-morphism such that there exists a $G$-equivariant invertible sheaf $\cM$ on $Q$ with $\lambda$-weight $d$ and $\cM|_{[Q_+/G]} \cong f^*\O_{\P_S(\cE)}(1)$.

 If $d \leq \mu$, the gauged Landau-Ginzburg model $( (U_{\cE^*}^1)_-, \widetilde{G}, \O(\beta), w)$ is a weak homological projective dual of $f$ with respect to the Lefschetz decomposition given by
 \[
  \cA_i = \begin{cases}
  \langle \dbcoh{[Q_- / G]}, \cZ_0^+, \ldots, \cZ_{d-1}^+ \rangle & i=0 \\
  \langle \cZ^+_0, \ldots, \cZ^+_{d-1} \rangle & 0 < i < \lceil\frac{\mu}{d} \rceil -1  \\
  \langle \cZ^+_0, \dots,\cZ^+_{\mu-d(\lceil\frac{\mu}{d}\rceil-1)} \rangle & i = \lceil\frac{\mu}{d} \rceil -1 \\
    \end{cases}
 \]
 In particular, there is a semi-orthogonal decomposition
 \begin{eqnarray*}
 \dbcoh{\mathcal X} = 
 \langle  \dcoh{[(U_{\cE^*}^1)_-/\widetilde{G}], w}, 
 \cZ^+_d \boxtimes \dbcoh{\P_S(\cE^*)}, \ldots \\ \ldots, \cZ^+_{\mu-1} \boxtimes \dbcoh{\P_S(\cE^*)} \rangle,
 \end{eqnarray*} where $\cZ^+_{k}$ are each equivalent to $\dbcoh{[Z_{\lambda}^0/C(\lambda)]}_k$.
\end{theorem}

Further below we will give a proof of this theorem, based on Theorem \ref{thm: bfkvgitlg} applied directly to the elementary wall-crossing given by $\lambda_1$ and Theorem \ref{thm: isik}, 
but we first state the other main result of this section.


Assume that $Q_- = \emptyset$, then we have, by Remark \ref{rem:simplification} below, that the Landau-Ginzburg model $$( (U_{\cE^*}^1)_-, \widetilde{G}, \O(\beta), w)$$ simplifies to the Landau-Ginzburg model
\begin{equation*}
(Q\times_S \P_S(\cE^*),G,\cM\boxtimes \cO_{\P_S(\cE^*)}(1),w).
\end{equation*}
In this case, the following theorem shows that the semi-orthogonal decompositions in the statement of Kuznetsov's Fundamental Theorem of Homological Projective Duality hold.

\begin{theorem} \label{thm: MainHPDMF}
  With the same assumptions as in Theorem \ref{thm: HPDMF} above, and assuming further that $Q_- = \emptyset$, we have the following  
\begin{itemize}

\item The derived category of the gauged Landau-Ginzburg model $(Q\times_S \P_S(\cE^*),G,\cM\boxtimes \cO_{\P_S(\cE^*)}(1),w)$  admits a dual Lefschetz collection
  \[ \dcoh{[Q\times_S \P_{S}(\cE^*)/G], w)}=\langle\mathcal B_{N -1}(-N +1),\ldots,\mathcal B_1(-1),\mathcal B_0\rangle, \]
  where $N$ is the rank of $\cE$, and with
  \[
  \cB_i = \begin{cases}
  \langle \cZ_0^+, \ldots, \cZ_{d-1}^+ \rangle & 0 \leq i \leq N - \lceil\frac{\mu}{d} \rceil -1 \\
  \langle  \cZ^+_{\mu+1-d(\lceil\frac{\mu}{d}\rceil-1)}, \ldots, \cZ^+_{d-1} \rangle & i = N - \lceil\frac{\mu}{d} \rceil\\
  0 & N - \lceil\frac{\mu}{d} \rceil < i < N \\
    \end{cases}
 \]

 \item Let $\cV = \cE^*/\cU $ be a quotient bundle of $\cE^*$ and $\cW = \cE/\cU^\perp$ be the corresponding quotient bundle of $\cE$. Assume that $X \times_{\P_S(\cE)} \P_S(\cW)$ is a complete linear section, i.e.
\[
\op{dim}(X \times_{\P_S(\cE)} \P_S(\cW)) = \op{dim}(X)-r,
\]
where $r$ is the rank of $\cU$.
Then,
\begin{itemize}
  \item[$\cbullet$]If $r < \lceil \frac{\mu}{d} \rceil $, there is a semi-orthogonal decomposition
    \begin{eqnarray*}
    \dbcoh{X \times_{\P_S(\cE)} \P_S(\cW)}=\langle\dcoh{[Q\times_S \P_{S}(\cV)/G], w)},\mathcal A_r(1),\ldots \\ \ldots,\mathcal A_i(i-r+1)\rangle,
    \end{eqnarray*}
\item[$\cbullet$] If  $r \geq \lceil \frac{\mu}{d} \rceil $, there is a semi-orthogonal decomposition
    \begin{eqnarray*}
      \dcoh{[Q\times_S \P_{S}(\cV)/G], w)}=\langle\mathcal B_{N-1}(-r-N -2),\ldots\\ \ldots,\mathcal B_{N -1-r}(-1),\dbcoh{X \times_{\P_S(\cE)} \P_S(\cW)}\rangle.
    \end{eqnarray*}

\end{itemize}
\end{itemize}
 
\end{theorem}

\begin{remark}
 The notations $\cZ^+_j$ and $\mathcal C_{\cV}$, are used to illustrate that the corresponding categories are equivalent, even though they are embedded by different functors and in different categories. 
\end{remark}

\begin{center}
\begin{figure}[h]
\tiny
\begin{tikzpicture}


\filldraw[fill=black!20!white,draw=white!100]
 (\xMin+ 3, \yMaxx) -- (\xMin+\xShift, \yMaxx) --(\xMin+\xShift,\yMin+\yShift) -- (\xMin+\xShift+1, \yMin+\yShift) -- (\xMin+\xShift+1, \yMin) -- (\xMin+3, \yMin) --  (\xMin+3, \yMaxx-1);

    \foreach \i in {\xMin,...,\xMax} {
        \draw [thin,black] (\i,\yMin) -- (\i,\yMaxx)  node [below] at (\i,\yMin) {};
    }
    \foreach \i in {\yMinshift,...,\yMaxshift} {
        \draw [thin,black,dashed] (\xMin,\i) -- (\xMax,\i) node [left] at (\xMin,\i) {};
    }

    \foreach \i in {\yMin,\yMaxx} {
        \draw [thin,black] (\xMin,\i) -- (\xMax,\i) node [left] at (\xMin,\i) {};
    }
    \node at (-.2, \yMin +.5) {$\cZ_{d-1}$};
               \node at (-.2, \yMin +\yShift - .5) {$\cZ_{\mu-d(\lceil\frac{\mu}{d}\rceil-1)}$}; 
        
            \node at (-.2, \yMaxx -1.5) {};
    \node at (-.2, \yMaxx -.5) {$\cZ_{0}$};

\node at (\xMin+.5, \yMin - .5) {$\cA_0$};
\node at (\xMin+.5 + \xShift, \yMin - .5) {$\cA_{\lceil\frac{\mu}{d} \rceil -1}$};

\node at (\xMax - .5, \yMaxx + .5) {$\cB_0$};
\node at (\xMin+.5 + \xShift, \yMaxx + .5) {$\cB_{N-\lceil\frac{\mu}{d} \rceil }$};
\node at (\xMin+.5, \yMaxx +.5) {$\cB_{N}$};

\draw [red, very thick] (\xMin,\yMaxx) -- (\xMin+\xShift, \yMaxx) -- (\xMin+\xShift, \yMaxx-1) --(\xMin+\xShift,\yMin+\yShift) -- (\xMin+\xShift+1, \yMin+\yShift) -- (\xMin+\xShift+1, \yMin) -- (\xMax, \yMin);

\draw[<->, blue, very thick] (4, \yMin-.5) -- (4, \yMaxx+.5);

\end{tikzpicture}
\normalsize
\captionof{figure}{A visual representation of the components appearing in the semi-orthogonal decompositions in Theorem \ref{thm: MainHPDMF}} \label{fig: extraspecial HPD table}
\end{figure}
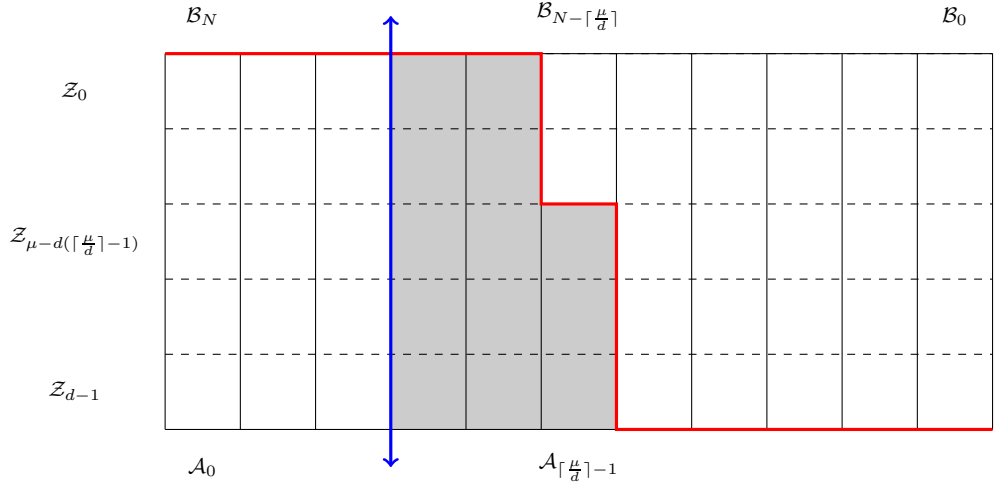
\end{center}

\begin{remark}\label{rem: HPDTableRemark}
  Figure~\ref{fig: extraspecial HPD table} demonstrates the tabular representation of the Fundamental Theorem of Homological Projective Duality in the case of the above theorem. The case when $r < \lceil \frac{\mu}{d} \rceil $ is pictured. The long vertical line is placed to the immediate left of $\cA_r$. Had $r \geq \lceil \frac{\mu}{d} \rceil $, the vertical line would have been to the right of $\cA_{\lceil \frac{\mu}{d} \rceil-1}$. The boxes between the staircase and the vertical line are highlighted. These
 correspond to the additional terms appearing in the decompositions of Theorem \ref{thm: MainHPDMF}. Comparing with Figure \ref{fig: HPD Table} and Remark \ref{rem: HPD table}, the situation here is simpler because in the case when $Q_- = \emptyset$, the Lefchetz decomposition is almost rectangular. 
\end{remark}

Before proving Theorems~\ref{thm: HPDMF} and~\ref{thm: MainHPDMF} we will set up a more complete picture of the various elementary wall-crossings that appear in the proofs. For each quotient bundle $\cV$ of $\cE^*$, we will set up what is, in principle, a variation of GIT quotients problem (we will specify four different elementary wall crossings arising in such a setup), which interpolates between the corresponding linear sections of $X$, $\cX$, the Landau-Ginzburg model $( (U_{\cE^*}^1)_-, \widetilde{G}, \O(\beta), w)$ which is the homological projective dual and the Landau-Ginzburg model whose derived category is equivalent to the category $\cC_V$ in the statement of Theorem \ref{thm: MainHPDMF}. The proofs will then follow from applying Theorem \ref{thm: bfkvgitlg} to some of these wall-crossings. 

Consider the variety
\[
\wtQ:= \op{V}_Q(\cM \oplus p^*\cV) = \op{V}_Q(\cM) \times_S \op{V}_S(\cV),
\]
equipped with a $ \widetilde{G} := G \times \gm\times \gm$-action described by
\begin{center}
\begin{tabular}{c c c}
   $\,$  & $\op{V}_Q(\cM)$ & $\op{V}_S(\cV)$  \\
   $G$ &   $g$    &  $1$ \\ 
   $\GG_m$  & $\alpha_1^{-1}$  & $\alpha_1$ \\
   $\GG_m$  & $\alpha_2$ &  $1$
\end{tabular}
\end{center}
The action of $G$ on $\op{V}_Q(\cM)$ is given by the $G$-equivariant structure on $\cM$ and this action is trivial on the $\op{V}_S(\cV)$ component.  The first $\gm$ acts with weight $-1$ on the fibers of $\op{V}_Q(\cM)$ and with weight $1$ on the fibers of $\op{V}_S(\cV)$. The second $\gm$ acts by dilation only on the fibers of $\op{V}_Q(\cM)$.

To describe the elementary wall-crossings we consider, we define four $\widetilde{G}$-invariant open subsets
\begin{align}
  \label{eq:u1}U_{\cV}^1 & := \wtQ \backslash ( \op{V}_Q(\cM) \times_S \mathbf{0}_{\op{V}_S(\cV)})\\ 
U_{\cV}^2 & := \wtQ \backslash( S_{-\lambda} \times_Q \mathbf{0}_{\op{V}_Q(\cM)} \times_S \op{V}_S( \cV) \cup S_{-\lambda} \times_Q {\op{V}_Q(\cM)} \times_S  \mathbf{0}_{\op{V}_S( \cV)})  \\ 
U_{\cV}^3 & := \wtQ \backslash(\mathbf{0}_{\op{V}_Q(\cM) }\times_S  \op{V}_S(\cV)  ) \\
 \label{eq:u4}U_{\cV}^4 & := \wtQ \backslash ( \op{V}_{{S_\lambda}}(\cM) \times_S \op{V}_S( \cV))
\end{align} 
of $\wtQ$, 
and one-parameter subgroups given by
\begin{align*}
  \lambda_1(\alpha) & := (\lambda(\alpha) , 1,1) \\
\lambda_2(\alpha) & := (1_G,\alpha,1) \\
\lambda_3(\alpha) & := \lambda_1(\alpha) \lambda_2(\alpha)^{d} = (\lambda(\alpha), \alpha^d, 1) \\
\lambda_4(\alpha) & := \lambda_2(\alpha)
\end{align*}

It is convenient to picture the elementary wall crossings we will describe as in Figure~\ref{fig: GIT fan}. Had these elementary wall-crossings corresponded to wall-crossings coming from varying a linearization, the GIT fan would look as in this figure. 

\begin{center}
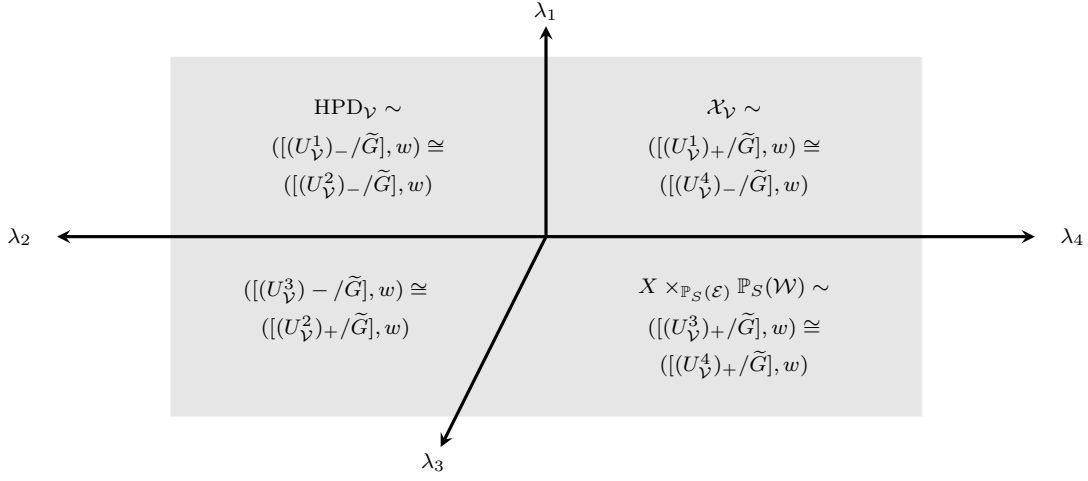
\begin{figure}
\tiny
\begin{tikzpicture}
  [scale=1, vertex/.style={circle,draw=black!100,fill=black!100,thick, inner sep=0.5pt,minimum size=0.5mm}, cone/.style={->,very thick,>=stealth}]
  \filldraw[fill=black!10!white,draw=white!100]
    (-5,2.4) -- (5,2.4) -- (5,-2.4) -- (-5,-2.4) -- (-5,2.4);
  \draw[cone] (0,0) -- (6.5,0);
  \draw[cone] (0,0) -- (0,2.8);
   \draw[cone] (0,0) -- (-6.5,0);
  \draw[cone] (0,0) -- (-1.4,-2.8);
  
  \node at (7, 0) {$\lambda_4$};
  \node at (-7.0, 0) {$\lambda_2$};
    \node at (0, 3) {$\lambda_1$};
      \node at (-1.5,-3) {$\lambda_3$};

          \node at (2.5, 1.7) {$\cX_{\cV} \sim$};
                \node at (2.5, 1.2) {$([(U_{\cV}^1)_+/\widetilde{G}], w) \cong$};
      \node at (2.5, .7) {$([(U_{\cV}^4)_-/\widetilde{G}], w)$};
         
          \node at (-2.5, 1.7) {$\text{HPD}_{\cV} \sim$};
                \node at (-2.5, 1.2) {$([(U_{\cV}^1)_-/\widetilde{G}], w) \cong$};
      \node at (-2.5, .7) {$([(U_{\cV}^2)_-/\widetilde{G}], w)$};

    \node at (-2.8, -.7) {$ ([(U_{\cV}^3)-/\widetilde{G}], w) \cong$};
  \node at (-2.8, -1.2) {$([(U_{\cV}^2)_+/\widetilde{G}], w)$};

          \node at (2.5, -.7) {$X \times_{\P_S(\cE)} \P_S(\cW) \sim$};
                \node at (2.5, -1.2) {$([(U_{\cV}^3)_+/\widetilde{G}], w) \cong$};
      \node at (2.5, -1.7) {$([(U_{\cV}^4)_+/\widetilde{G}], w)$};
\end{tikzpicture}
\normalsize
\captionof{figure}{Hypothetical GIT fan relating categories appearing in HPD} \label{fig: GIT fan}
\end{figure}
\end{center}

To clarify, in what follows, we set
\begin{displaymath}
 (U_{\cV}^i)_{\pm} := U_{\cV}^i \backslash  S_{\pm\lambda_i}.
\end{displaymath}
The explicit formulas for the $S_{\lambda_i}$ can be obtained by comparing the following lemma with equations \ref{eq:u1}-\ref{eq:u4}. 

\begin{lemma}
There are equalities 
\begin{align} 
  (U_{\cV}^3)_- & = (U_{\cV}^2)_+ ={\widetilde Q}_{\cV}\backslash(\mathbf{0}_{\op{V}_Q(\cM)} \times_S \op{V}_S(\cV)\cup \op{V}_{S_{-\lambda}}(\cM)\times_S\mathbf{0}_{\op{V}_S(\cV)} ) \label{eq: comparison 1}\\
(U_{\cV}^1)_- & = (U_{\cV}^2)_- = {\widetilde Q}_{\cV}\backslash(\op{V}_Q(\cM) \times_S \mathbf{0}_{\op{V}_S(\cV)}\cup \mathbf{0}_{\op{V}_{S_{-\lambda}}(\cM)}\times_S\op{V}_S(\cV)) \label{eq: comparison 2} \\
(U_{\cV}^3)_+ & = (U_{\cV}^4)_+ = {\widetilde Q}_{\cV}\backslash(\mathbf{0}_{\op{V}_Q(\cM)} \times_S \op{V}_S(\cV)\cup \op{V}_{S_{\lambda}}(\cM)\times_S\op{V}_S(\cV) )  \\
(U_{\cV}^1)_+ & = (U_{\cV}^4)_- ={\widetilde Q}_{\cV}\backslash(\op{V}_{S_{\lambda}}(\cM) \times_S \op{V}_S(\cV)\cup \op{V}_Q(\cM)\times_S\mathbf{0}_{\op{V}_S(\cV)} ) 
\end{align}
\end{lemma}
\begin{proof}
This is easily checked.
\end{proof}

\begin{lemma} \label{lem: new crossing}
There are new elementary wall-crossings $(\mathfrak{K^+}_i,\mathfrak{K^-}_i)$ for $1 \leq i \leq 4$.
 \begin{align*}
U_{\cV}^i & = (U_{\cV}^i)_+ \sqcup S_{\lambda_i}\\
U_{\cV}^i & = (U_{\cV}^i)_- \sqcup{S}_{-\lambda_i} \\
 \end{align*}
with
 \[
 t(\mathfrak{K^+}_i) = 
 \begin{cases} 
 t(\mathfrak{K^+}) & \emph{if } i=1,3 \\
\emph{rk }\cV & \emph{if } i=2, 4
 \end{cases}
 \]
 and
  \[
 t(\mathfrak{K^-}_i) = 
 \begin{cases} 
 t(\mathfrak{K^-}) - d & \emph{if } i=1 \\
 t(\mathfrak{K^-}) - d \cdot  \emph{rk }\cV  & \emph{if } i=3 \\
1 & \emph{if } i=2,4 
 \end{cases}
 \]
\end{lemma}

\begin{proof}
We treat the case where $i=1$. The rest are similar. Denote by
\[
i_\pm : Q_\pm \to Q
\]
the open immersions. Notice that
 \begin{align*}
   U_{\cV}^1& = (\op{V}_{Q_+}(i_+^* \cM)) \times_S  (\op{V}_S(\cV) \backslash \mathbf{0}_{\op{V}_S(\cV)}) \, \sqcup \, \op{V}_{S_\lambda}(\cM|_{{S}_{\lambda}}) \times_S (\op{V}_S(\cV) \backslash \mathbf{0}_{\op{V}_S(\cV)})  \\
   U_{\cV}^1& = (\op{V}_{Q}(\cM) \backslash \mathbf{0}_{\op{V}_{S_{-\lambda}}(\cM|_{{S}_{-\lambda}})})\times_S  (\op{V}_S(\cV) \backslash \mathbf{0}_{\op{V}_S(\cV)}) \, \sqcup \, \mathbf{0}_{\op{V}_{S_{-\lambda}}(\cM|_{{S}_{-\lambda}})}\times_S (\op{V}_S(\cV) \backslash \mathbf{0}_{\op{V}_S(\cV)})  \\
 \end{align*}
We will verify that these are elementary HKKN stratifications. 

 As $S_{\pm \lambda}$ are closed by assumption, it is clear that  
 \begin{align*}
  S_{\lambda_1} & = \op{V}_{S_\lambda}(\cM|_{{S}_{\lambda}}) \times_S (\op{V}_S(\cV) \backslash \mathbf{0}_{\op{V}_S(\cV)}) \\
  S_{-\lambda_1} & = \mathbf{0}_{\op{V}_{S_{-\lambda}}(\cM|_{{S}_{-\lambda}})}\times_S (\op{V}_S(\cV) \backslash \mathbf{0}_{\op{V}_S(\cV)}) 
 \end{align*}
 are closed in $U_{\cV}^1$. Furthermore,
 \begin{align*}	
  Z_{\lambda_1} & = \op{V}_{Z_{\lambda}}(\cM|_{{Z}_{\lambda}}) \times_S (\op{V}_S(\cV) \backslash \mathbf{0}_{\op{V}_S(\cV)}) \\
  Z_{-\lambda_1} & =  \mathbf{0}_{\op{V}_{Z_{-\lambda}}(\cM|_{{Z}_{-\lambda}})}\times_S (\op{V}_S(\cV) \backslash \mathbf{0}_{\op{V}_S(\cV)}).
 \end{align*}
By assumption
 \[
 \tau_{\lambda}: [G \times Z_{\lambda}/P(\lambda)] \to S_{\lambda}
 \]
is an isomorphism.  Also
\[
P(\pm\lambda_1) = P(\pm \lambda) \times \gm \times \gm.
\]
  
It remains to check that the maps
\[
\tau_{\lambda_1}: [(G \times \gm\times \gm) \times \op{V}_{Z_{\lambda}}(\cM|_{{Z}_{\lambda}}) \times_S (\op{V}_S(\cV) \backslash \mathbf{0}_{\op{V}_S(\cV)}) / P(\lambda_1)] \to \op{V}_{S_\lambda}(\cM|_{{S}_{\lambda}}) \times_S (\op{V}_S(\cV) \backslash \mathbf{0}_{\op{V}_S(\cV)}) 
\]
and
\[
\tau_{-\lambda_1}: [(G \times \gm\times \gm) \times \mathbf{0}_{\op{V}_{Z_{-\lambda}}(\cM|_{{Z}_{-\lambda}})}\times_S (\op{V}_S(\cV) \backslash \mathbf{0}_{\op{V}_S(\cV)}) 
 / P(-\lambda_1)] \to \mathbf{0}_{\op{V}_{S_{-\lambda}}(\cM|_{{S}_{-\lambda}})}\times_S (\op{V}_S(\cV) \backslash \mathbf{0}_{\op{V}_S(\cV)})
 \]
are isomorphisms.

We will check this for the first map; the proof for the second one is similar. First, we can cancel the $\gm\times \gm$ with the one appearing in $P(\lambda_1) = P(\lambda) \times \gm\times \gm$ and look at the map
 \[
  [G \times \op{V}_{Z_{\lambda}}(\cM|_{{Z}_{\lambda}}) \times_S (\op{V}_S(\cV) \backslash \mathbf{0}_{\op{V}_S(\cV)}) / P(\lambda)] \to \op{V}_{S_\lambda}(\cM|_{{S}_{\lambda}}) \times_S (\op{V}_S(\cV) \backslash \mathbf{0}_{\op{V}_S(\cV)}) 
 \]
 Now, we can forget the $(\op{V}_S(\cV)\backslash \mathbf{0}_{\op{V}_S(\cV)})$ on both sides, as $P(\lambda)$ acts trivially on this factor, and look at the map
 \[
  [G \times \op{V}_{Z_{\lambda}}(\cM|_{{Z}_{\lambda}}) / P(\lambda)] \to \op{V}_{S_\lambda}(\cM|_{{S}_{\lambda}})  
 \]
 or equivalently
 \[
  [\op{V}_{G\times Z_{\lambda}}(\O_G \boxtimes \cM|_{G \times {Z}_{\lambda}}) / P(\lambda)] \to  \op{V}_{S_\lambda}(\cM|_{{S}_{\lambda}}) .
 \]
 We have an isomorphism, $\tau_{\lambda}^*\mathcal M|_{S_{\lambda}} \cong \mathcal O_{G} \boxtimes \mathcal M|_{Z_{\lambda}}$. This induces the desired isomorphism on the corresponding geometric vector bundles. 
 
 For the computation of $t(\mathfrak{K^{+}}_1)$, observe that the relative canonical bundle $\omega_{S_{\lambda_1}/U_1}$ is the pullback of $\omega_{S_{\lambda}/Q}$ to $S_{\lambda_1}$.
 Since $\lambda_1 = (\lambda,1,1)$, the $\lambda_1$-weight of $\omega_{S_{\lambda_{1}}/U_{1}}|_{Z^0_\lambda}$ is the same as the $\lambda$-weight of $\omega_{S_{\lambda}/Q}|_{Z^0_\lambda}$. Therefore, $t(\mathfrak{K^{+}}_1) = t(\mathfrak{K^{+}})$. For $t(\mathfrak{K^{-}}_1)$, observe that $\omega_{S_{-\lambda_1}/U_{1}}$ is the pullback of $\omega_{S_{-\lambda}/Q}$ tensored with $\cM^{*}$. Therefore, we have $t(\mathfrak{K^{-}}_1) = t(\mathfrak{K^{-}})-d$. 
 \end{proof}

\begin{remark}
The fourth elementary wall crossing corresponding to $U_{\cV}^4$ and $\lambda_4$ is not used in the proofs which follow.  However, it is interesting to note that this wall crossing can be used to prove the semi-orthogonal decompositions appearing in the Fundamental Theorem of Homological Projective Duality in the case where the Lefschetz collection is the trivial one with $\cA_0 = \dbcoh{X}$, which would give that $$\dbcoh{\cX_L}=\langle\dbcoh{X_L},\dbcoh{X},\ldots,\dbcoh{X}\rangle.$$
\end{remark}

As we noted above, $\cX$, the relative universal hyperplane section of the $S$-morphism $f: X \to \P_S(\cE)$, is the zero locus of the pullback of the canonical section, $\theta_{\cE} \in \Gamma( \P_S(\cE) \times_S \P_S(\cE^*) , \mathcal O(1,1))$, to $X \times_S \P(\cE^*)$.  Furthermore, we constructed a unique $G\times\gm$-invariant function on $[(U_{\cE^*}^1)_+]$ that corresponds to $\theta_{\cE}$. Since $\Gamma(\widetilde{Q}_{\cE^*}, \O)^{G \times \gm}  \cong \Gamma([(U_{\cE^*}^1)_+ /G \times \gm], \O)$  we observe that there exists a unique $w$, a $G \times \gm$-invariant function on $\widetilde{Q}_{\cE^*}$, corresponding to the canonical section 
\[
(\theta_{\cE}: \O \to \cE \times \cE^*) \in  \Gamma(S, \op{Sym}^1(\cE \otimes_{\O_S} \cE^*)).
\]
Moreover, $w$ has weight $1$ with respect to the third factor, therefore $w$ is a semi-invariant function with respect to the $\widetilde{G}$-action with character $\beta(g,\alpha_1,\alpha_2) = \alpha_2$. In other words, $w$ corresponds to a section in $\Gamma(\widetilde{Q}_{\cE^*}, \cO(\beta))^{\widetilde{G}}$. 
We can apply the same construction for $\cV$ instead of $\cE^*$ and we will abuse notation by also writing this section as $w$ when $U_{\cV}^i$ is an open subset of $\wtQ$ for general $\cV$ even though $w$, in general, depends on both $\cV$ and $1 \leq i \leq 4$.


\begin{proof}[Proof of Theorem~\ref{thm: HPDMF}]

We will prove the statement for any quotient bundle $\cV$ of $\cE^*$ and then, setting $\cV=\cE^*$, we will obtain the desired result.

Consider the gauged Landau-Ginzburg model, $(U_{\cV}^1, \widetilde{G}, \O(\beta), w)$ as above. By Theorem~\ref{thm: isik}, there is an equivalence
\[
\dbcoh{\cX_{\cV}} \cong \dcoh{[(U_{\cV}^1)_+/ \widetilde{G}], w}.
\]
Consider the elementary wall-crossing $((\mathfrak K^+)_1, (\mathfrak K^-)_1)$ from Lemma~\ref{lem: new crossing}.  
Notice that since, by assumption, the $G$-action has weight $d>0$  on the fibers of $\op{V}_Q(\cM)$, we can choose the following connected component of the fixed locus of $\lambda_1$:
$$ Z_{\lambda_1}^0  :=  (\mathbf{0}_{\op{V}_{Z_{\lambda}^0}(\cM|_{Z^0_{\lambda}})} \times_S \op{V}_S(\cV)) \cap U_{\cV}^1$$
where $Z_{\lambda}^0$ is the connected component of the fixed locus chosen for $(\mathfrak K^+, \mathfrak K^-)$.  
Finally, inside $\widetilde{G}$ we have, 
\[
C(\lambda_1) = C(\lambda) \times \gm \times \gm
\]
and 
\[
[Z_{\lambda_1}^0 / C(\lambda_1)]  \cong [Z_{\lambda}^0 / C(\lambda) \times \gm] \times_S \P_S(\cV)
\] 
 where the $\gm$-action is trivial.
 Furthermore, for this choice, 
\begin{displaymath}
{S}_{\lambda_1}^0 = \op{V}((\cM \oplus p^* \cV)|_{S^0_{\lambda}}) \cap U_{\cV}^1
\end{displaymath}
which admits a $G$-invariant affine cover as we have assumed the existence of such for $S^0_{\lambda}$.

Therefore, we may apply Theorem~\ref{thm: bfkvgitlg} to obtain a weak semi-orthogonal decomposition
\[
 \dbcoh{\mathcal X \times_{\P_S(\cE)} \P_S(\cW)} = \langle  \cZ^+_0 \boxtimes \dbcoh{\P_S(\cV)}, \ldots, \cZ^+_{\mu-1} \boxtimes \dbcoh{\P_S(\cV)}, \dcoh{[(U_{\cV}^1)_-/ \widetilde{G}], w} \rangle.
\]
As $\mathcal X$ is smooth and proper, by \cite[Corollary 3.1.5]{BvdB} and \cite[Proposition 2.8]{BK}, we have a semi-orthogonal decomposition and can mutate to get a new semi-orthogonal decomposition
\[
 \dbcoh{\mathcal X_{\cV}} = \langle  \dcoh{[(U_{\cV}^1)_-/ \widetilde{G}], w}, \cZ^+_0 \boxtimes \dbcoh{\P_S(\cV)}, \ldots, \cZ^+_{\mu-1} \boxtimes \dbcoh{\P_S(\cV)} \rangle.
\]
This identifies $\mathcal D_{\cV}$ with $ \dcoh{[(U_{\cV}^1)_-/ \widetilde{G}], w}$. Taking $\cV = \cE^*$, we get the desired semi-orthogonal decomposition.

The $\P_S(\cE^*)$-linearity of the fully faithful functor $\dcoh{[(U_{\cE^*}^1)_-/ \widetilde{G}], w}\ra \dbcoh{\mathcal X}$ follows from the linearity of all the functors involved. 
\end{proof}

Since, in the statement of Theorem \ref{thm: MainHPDMF}, we assume that $Q_{-} = \emptyset$, the picture is simpler. In this case, we have $(U_{\cV}^2)_+ = (U_{\cV}^2)_-$ and the linear sections of the weak homological projective dual can be directly compared to $X \times_{\P_S(\cE)} \P_S(\cW)$.  Had these elementary wall-crossings corresponded to wall-crossings coming from varying a linearization, the GIT fan would look as in Figure~\ref{fig: GIT fan special case}.  

\begin{center}
\begin{figure}[h]
\tiny
\begin{tikzpicture}
  [scale=1, vertex/.style={circle,draw=black!100,fill=black!100,thick, inner sep=0.5pt,minimum size=0.5mm}, cone/.style={->,very thick,>=stealth}]
  \filldraw[fill=black!10!white,draw=white!100]
    (-5,2.4) -- (5,2.4) -- (5,-2.4) -- (-5,-2.4) -- (-5,2.4);
  \draw[cone] (0,0) -- (6.5,0);
  \draw[cone] (0,0) -- (0,2.8);
  \draw[cone] (0,0) -- (-1.4,-2.8);
  
  \node at (7, 0) {$\lambda_1$};
    \node at (0, 3) {$\lambda_4$};
      \node at (-1.5,-3) {$\lambda_3$};
      
        \node at (2.5, 1.7) {$\cX \times_{\P_S(\cE)} \P_S(\cW) \sim $};
               \node at (2.5,1.2) {$ ([(U_{\cV}^1)_+/ \widetilde{G}], w) \cong $};
               \node at (2.5,.7) {$ ([(U_{\cV}^4)_-/ \widetilde{G}], w)$};
        
          \node at (-2.5, 1.2) {$\text{HPD}_{\cV} \sim$};
                \node at (-2.5, 0.8) {$([(U_{\cV}^1)_-/\widetilde{G}], w) \cong$};
                \node at (-2.5, 0.3) {$([(U_{\cV}^2)_-/\widetilde{G}], w) \cong$};
                \node at (-2.5, -0.2) {$([(U_{\cV}^2)_+/\widetilde{G}], w) \cong$};
      \node at (-2.5, -0.7) {$([(U_{\cV}^3)_-/\widetilde{G}], w)$};

  \node at (2.5, -.7) {$X \times_{\P_S(\cE)} \P_S(\cW) \sim$};      
    \node at (2.5, -1.2) {$([(U_{\cV}^3)_+/\widetilde{G}], w) \cong$};
    \node at (2.5, -1.7) {$([(U_{\cV}^4)_+)/\widetilde{G}], w)$};

\end{tikzpicture}
\normalsize
\captionof{figure}{Hypothetical GIT fan relating categories appearing in HPD in the case $Q_- = \emptyset$} \label{fig: GIT fan special case}
\end{figure}
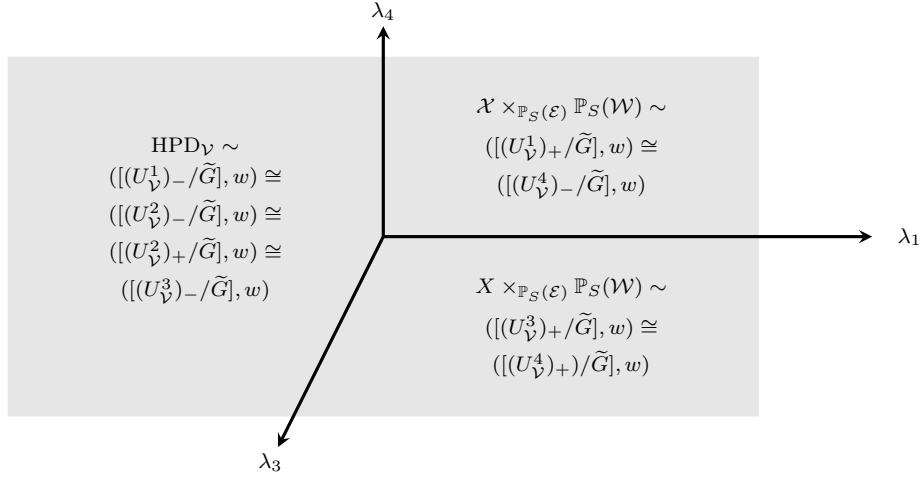
\end{center}

\begin{remark}
  \label{rem:simplification}
In this case, we can simplify the weak homological projective dual $$( (U_{\cE^*}^1)_-, \widetilde{G}, \O(\beta), w)$$ to the Landau-Ginzburg model 
\begin{equation*}
(Q\times_S \P_S(\cE^*),G,\cM\boxtimes \cO_{\P_S(\cE^*)}(1),w),
\end{equation*} 
as well as its linear sections, by cancelling the $\gm$-actions which are now free. That is, 
for any quotient bundle $\cV$ of $\cE^*$,  the Landau-Ginzburg models 
$$( (U_{\cV}^1)_-, \widetilde{G}, \O(\beta), w)$$ and  
$$(Q\times_S \P_S(\cV),G,\cM\boxtimes \cO_{\P_S(\cV)}(1),w)$$ have equivalent derived cateogies.
\end{remark}

\begin{remark}
  Although the second elementary wall crossing corresponding to $U_{\cV}^2$ and $\lambda_2$ is not used in the proof below, we have chosen to keep it in the main construction (see Figure~\ref{fig: GIT fan special case}), as it still allows one to construct semiorthogonal decompositions similar to the ones in Theorem \ref{thm: MainHPDMF}, albeit not in the correct order.
\end{remark}

\begin{proof}[Proof of Theorem~\ref{thm: MainHPDMF}]


  We consider $U^{3}_{\cV}$ and the elementary wall-crossing $((\mathfrak K^+)_3, (\mathfrak K^-)_3)$ of Lemma~\ref{lem: new crossing}. The fixed locus of $\lambda_3$ is 
\[
  Z^0_{\lambda_3} = \left( \op{V}_{Z_{\lambda}^0}(\cM|_{Z^0_{\lambda}}) \backslash {\mathbf 0}_{\op{V}_{Z_{\lambda}^0}(\cM|_{Z^0_{\lambda}})}  \right) \times_S \mathbf{0}_{\op{V}_S(\cV)}.
\] 
and since $C(\lambda_3) = C(\lambda) \times \gm \times \gm$ we may cancel the fibers of this line bundle with the first $\gm$-action to obtain an isomorphism
\[
[Z^0_{\lambda_3} / C(\lambda_3) ] \cong [Z^0_{\lambda} / C(\lambda) \times \gm]
\]

We apply Theorem~\ref{thm: bfkvgitlg} and Proposition~\ref{cor: isik} to two cases. When $\mu > dr$ we obtain a semi-orthogonal decomposition
\begin{equation} \label{eq: hat1}
\dcoh{[(U_{\cV}^3)_+/\widetilde{G} ] , w} = \langle \dcoh{[(U_{\cV}^3)_- /\widetilde{G}], w}, \cZ^+_0, \ldots, \cZ^+_{\mu-dr-1}\rangle.
\end{equation}
When $\mu \leq dr$ we obtain a semi-orthogonal decomposition
\begin{equation} \label{eq: hat2}
\dcoh{[(U_{\cV}^3))_-/\widetilde{G}] , w} = \langle \dcoh{[(U_{\cV}^3)_+ /\widetilde{G}], w}, \cZ^-_0, \ldots, \cZ^-_{dr-\mu-1}\rangle,
\end{equation}
where, as before, we denote by $\cZ^-_k$ the essential images of the fully faithful functors $\Upsilon^-_{k}$.

Therefore, in the case $\mu \leq dr$ we have 
\begin{equation}\label{eq: eq1}
\begin{aligned} 
  \dcoh{[(U_{\cV}^1)_-/\widetilde{G} ] , w} &  = \dcoh{[(U_{\cV}^3)_- /\widetilde{G}], w} \\ 
&  =  \langle \dcoh{[(U_{\cV}^3)_+ /\widetilde{G}], w}, \cZ^-_0, \ldots, \cZ^-_{dr-\mu-1}\rangle \\
&  =  \langle \dbcoh{X \times_{\P_S(\cE)} \P_S(\cW)}, \cZ^-_0, \ldots, \cZ^-_{dr-\mu-1}\rangle \\
&  =  \langle \dbcoh{X \times_{\P_S(\cE)} \P_S(\cW)}, \cZ^+_0, \ldots, \cZ^+_{dr-\mu-1}\rangle 
\end{aligned}
\end{equation}
where the the second line comes from \eqref{eq: hat2}, the third is by Theorem~\ref{thm: isik} and using the fact that $X \times_{\P_S(\cE)} \P_S(\cW)$ is a complete linear section, and the last comes from equivalences between the essential images $\cZ^+_j$ of $\Upsilon^+_j$ and the essential images $\cZ^-_j$ of $\Upsilon^-_j$. 

In the case where, $dr < \mu$, as before, considering the decomposition \eqref{eq: hat1} will suffice.

We can now proceed to the proof of the statements in the theorem. Setting $\cV=\cE^*$ and noticing that in this case $X \times_{\P_S(\cE)} \P_S(\cW)=\emptyset$ , we obtain the dual Lefschetz decomposition \[ \dbcoh{[(U_{\cE^*}^1)_-/ \widetilde{G}], w)}=\langle\mathcal B_{N -1}(-N +1),\ldots,\mathcal B_1(-1),\mathcal B_0\rangle \] where
  \[
  \cB_i = \begin{cases}
  \langle \cZ_0^+, \ldots, \cZ_{d-1}^+ \rangle & 0 \leq i \leq N - \lceil\frac{\mu}{d} \rceil -1 \\
  \langle \cZ^+_{\mu+1-d(\lceil\frac{\mu}{d}\rceil-1)}, \ldots, \cZ^+_{d-1} \rangle & i = N - \lceil\frac{\mu}{d} \rceil\\
  0 & N - \lceil\frac{\mu}{d} \rceil < i < N \\
    \end{cases}.
 \]

 Combined with Remark \ref{rem:simplification}, equation \ref{eq: eq1}, gives the statement of the theorem in the case where $\mu \leq dr$  and similarly equation \ref{eq: hat1} gives the statement of the theorem in the case where $dr < \mu$ (see Figure~\ref{fig: extraspecial HPD table} and Remark~\ref{rem: HPDTableRemark}).

\end{proof}

\begin{remark} \label{rem: derived scheme}
 Dropping the assumption that $X \times_{\P_S(\cE)} \P_S(\cW)$ is a complete linear section the theorem above continues to hold if we replace $X \times_{\P_S(\cE)} \P_S(\cW)$ either by the gauged Landau-Ginzburg model $(U_{\cV}^3, \widetilde{G}, \O(\beta), w)$ or equivalently by the derived fiber product $X \times_{\P_S(\cE)}^{\mathbb{L}} \P_S(\cW)$ (see Remark 4.7 of \cite{Isik}).
\end{remark}

\subsection{A first example: Projective Bundles}

In this section we provide an elementary and explicit example of Homological Projective Duality using the results of the previous section. The results presented here were first proved in \cite{KuzHPD}. 

 Let $\mathcal P$ be a locally-free coherent sheaf on $B$ with,
 \[
 V:= \op{H}^0(\mathcal P)^* \neq 0.
 \] 
  For the projective bundle, 
  \[
  \pi: \P_B(\mathcal P) \to B,
  \]
   the relative invertible sheaf, $\O_{\op{\P(\cP)}}(1)$, provides a a map,
  \[
  j: \P_B(\mathcal P) \to \P(V).
  \]
  With the notation as in the previous section, we set 
  \begin{align*}
  Q  & = \op{V}_B(\mathcal P), \\
  G & = \gm, \\
  \lambda(\alpha) & = \alpha^{-1}, \\
  \cM & = \O(\chi), \\
  S & =\text{Spec }k,
  \end{align*}
  where $\gm$ acts by fiber-wise dilation and $\chi(\alpha)=\alpha$.
 It follows that
  \begin{align*}
  [Q_+/G] & = \P_B(\mathcal P), \\
  [Q_-/G] & = \emptyset \\
  \mu & = \text{rk }\mathcal P, \\
  d & = 1, \\
  \mathcal A_s & = \pi^*\dbcoh{B} \text{ for } 0 \leq s < \mu
  \end{align*}
  
  By Remark~\ref{rem:simplification}, the weak homological projective dual reduces to
  \[
(\op{V}_B(\mathcal P) \times_k \P(V^*), \gm, \O(\chi) \boxtimes \O(1), w)
  \]
  with $\gm$ acting fiber-wise with weight $1$.  This is isomorphic to
  \[
(\op{V}_{B \times_k \P(V^*)}{(\mathcal P \boxtimes \O(1))}, \gm, \O(\chi), w) 
  \]
  where $\gm$ acts fiber-wise with weight $1$. Therefore, in this case, we can do more.  Namely, we may apply Theorem~\ref{thm: isik} to see that
  \[
  \dcoh{[\op{V}_{B \times_k \P(V^*)}{(\mathcal P \boxtimes \O(1))}/ \gm], w} \cong \dbcoh{Z(w)}
  \]
  where $Z(w)$ is the zero locus of $w$ in $B \times_k \P(V^*)$. Furthermore, by definition, $Z(w)$ can be described as the set
\[
Z(w) = \{ (b, s) | s(b)=0 \} \subseteq B \times_k \P(V^*).
\]

\begin{remark}
This is precisely the homological projective dual obtained by Kuznetsov in \cite{KuzHPD}.  Also notice, as observed in \cite[Lemma 8.1]{KuzHPD}, that
\[
Z(w) \cong \P_B(\mathcal P^{\perp})
\]
where $\mathcal P^{\perp}$ is the locally-free coherent sheaf defined as the kernel of the evaluation map
\[
V^* \otimes \O_B \to \mathcal P.
\]
\end{remark}

\begin{remark}
If we project down to $\P(V^*)$ then the fiber over $s \in V^* = \op{H}^0(B, \mathcal P)$ is precisely the vanishing of $s$.  In particular, the image is the set of degenerate sections of $\mathcal P$.  When $\text{rk }\mathcal P = \text{dim }B +1$, this is precisely the projective dual of  $\P_B(\mathcal P)$ (see Theorem 3.11 in \cite{GKZ}).  However, unlike the usual projective dual, the homological projective dual is smooth.
\end{remark}

\section{Homological Projective Duality for d-th degree Veronese embeddings}

We will now apply the results of the previous two sections to construct a homological projective dual to the degree $d$ Veronese embedding. In view of potential applications, we will do this in the relative setting. Let $S$ be a smooth, connected variety and $\mathcal P$ be a locally-free coherent sheaf on $S$. We consider the relative degree $d$ Veronese embedding for $d > 0$, 
  \begin{displaymath}
   g_d: \P_S(\mathcal P) \to \P_S(S^d\mathcal P).
  \end{displaymath}

Notice that $g_d^*(\cO_{\P({S^d\mathcal P})}(1))\cong \cO_{\P({\mathcal P})}(d)$. Consider the Lefschetz decomposition 
$$\dbcoh{\P_S(\mathcal P)}=\langle\cA_0,\ldots,\cA_i(i)\rangle$$
where the subcategories $\cA_j$ are defined to be
\begin{align*}
\cA_0=\ldots =\cA_{i-1}=\langle p^*\dbcoh{S},\ldots,p^*\dbcoh{S}(d-1)\rangle \\
\cA_i=\langle p^*\dbcoh{S},\ldots,p^*\dbcoh{S}(k-1)\rangle
\end{align*} 
where $k=\text{rk }\mathcal P-d(\lceil\frac{\text{rk }\mathcal P}{d} \rceil-1)$.

We will first consider $\P_S(\cP)$ as a quotient and use the results of Section \ref{sec:lghpd}. Let us take $Q=\op{V}_S(\mathcal P)$ and consider the $G=\gm$-action given by fiber-wise dilation. 
Take the character given by $\chi (\alpha)=\alpha^d$ and the invertible sheaf $\cM=\O(\chi)$ on $Q$. 
Taking the one-parameter subgroup $\lambda : \gm \to \gm$ given by $\lambda(\alpha) = \alpha^{-1}$, we see that we have an elementary wall crossing with
$$S_\lambda = \mathbf{0}_{\op{V}_S(\cP)},$$
$$S_{-\lambda} = \op{V}_S(\cP).$$
We get that
  \begin{align*}
  [Q_+/G] & = \P_S(\mathcal P), \\
  [Q_-/G] & = \emptyset \\
  \mu & = \text{rk }\mathcal P-d, \\
  d & = d, \\
    \end{align*} 
where $d$ is the weight of the $\lambda$-action on $\cM$. This shows that $\cM$ induces the morphism
$g_d : \P(\cP) \rightarrow \P(S^d\cP^*).$

Using Proposition \ref{prop: Lef dec from VGIT}, we recover the Lefschetz decomposition with
\begin{align*} 
\cA_0=\ldots =\cA_{i-1}=\langle p^*\dbcoh{S},\ldots,p^*\dbcoh{S}(d-1)\rangle , \\
\cA_i=\langle p^*\dbcoh{S},\ldots,p^*\dbcoh{S}(k-1)\rangle.
\end{align*}

The universal degree $d$ polynomial $w$ is given by
\begin{displaymath}
   w := (g_d \times 1)^* \theta \in \Gamma(\P_S(\mathcal P)) \times_S \P_S(S^d\mathcal P^*), \mathcal O_{\P_S(\mathcal P)}(d) \boxtimes \mathcal O_{\P_S(S^d\mathcal P^*)}(1)),
\end{displaymath}
where $\theta$ is the tautological section in $\Gamma( \P_S(S^d\mathcal P) \times_S \P(S^d \mathcal P^*),\mathcal O_{\P(S^d\cP)}(1) \boxtimes \mathcal O_{\P_S(S^d\mathcal P^*)}(1))$. 
The zero locus $w$ in $\P_S(\mathcal P) \times_S \P_S(S^d\mathcal P^*)$ is the universal hyperplane section $\cX$ of $\P_S(\mathcal P)$ with respect to the embedding $g_d$.

We have thus constructed a Landau-Ginzburg model which is a homological projective dual.

\begin{theorem-section}\label{prop:veroneselghpd}
  The gauged Landau-Ginzburg model $([\op{V}_S(\mathcal P)\times_S\P_S(S^d\mathcal P^*)/{\gm}],w)$ is a weak homological projective dual to $\P_S(\mathcal P)$ with respect to the embedding $g_d$ and the Lefschetz decomposition constructed above.

Moreover, we have:
\begin{itemize}
\item The derived category of the Landau-Ginzburg model $([\op{V}_S(\mathcal P)\times_S\P_S(S^d\mathcal P^*)/{\gm}],w)$ admits a dual Lefschetz collection
  \[ \dcoh{[\op{V}_S(\mathcal P)\times_S\P_S(S^d\mathcal P^*)/{\gm}],w} = \langle\mathcal B_{j}(-j),\ldots,\mathcal B_1(-1),\mathcal B_0\rangle \]

\item Let $\cV \subset (S^d\mathcal P^*)/\cU$ be a quotient bundle and $\cW = (S^d\mathcal P) / \cU^\perp$. Let $r = \op{rk}\cU$. Assume that $\P_S(\mathcal P) \times_{\P_S(S^d\mathcal P)} \P_S(\cW)$ is a complete linear section (not necessarily smooth). Then, there exist semi-orthogonal decompositions,
  \begin{itemize}
    \item[$\cbullet$] If $r < \lceil \frac{\op{rk}\cP - d}{d} \rceil - 1$, 
      \begin{eqnarray*}
\dbcoh{\P_S(\mathcal P) \times_{\P_S(S^d\mathcal P)} \P_S(\cW)}=\langle \dcoh{[\op{V}_S(\mathcal P)\times_S\P_S(\cV)/{\gm}],w } ,\mathcal A_r(1),\ldots \\ \ldots,\mathcal A_i(i-r+1)\rangle,
      \end{eqnarray*}
\item[$\cbullet$] If $r \geq \lceil \frac{\op{rk}\cP - d}{d} \rceil - 1$,
  \begin{eqnarray*}
   \dcoh{[\op{V}_S(\mathcal P)\times_S\P_S(\cV)/{\gm}],w } =\langle\mathcal B_{0}(-r-N -2),\ldots \\ \ldots,\mathcal B_{N -1-r}(-1),  \dbcoh{\P_S(\mathcal P) \times_{\P_S(S^d\mathcal P)} \P_S(\cW)} \rangle.
  \end{eqnarray*}
 \end{itemize}
\end{itemize}

\end{theorem-section}
\begin{proof}
  By directly applying Theorems \ref{thm: HPDMF} and \ref{thm: MainHPDMF} to the elementary wall crossing described above, and simplifying as described in Remark \ref{rem:simplification}, we get the desired result. 
\end{proof}

\begin{remark-section}  For the first part of the theorem, we can alternatively consider $\cX$ as a degree $d$ hypersurface fibration over $\P(S^d\cP^*)$ and use a relative version of Orlov's theorem, proven in \cite{BDFIKdegreed}, with $S=\P_S(S^d\mathcal P^*)$, $\mathcal E=\pi^*\mathcal P$ and $\mathcal U=\O_{\P_S(S^d\mathcal P^*)}(1)$ to get the decomposition 
\begin{align*}
& \dbcoh{\mathcal X}= \\
& \,\,\,\,\, \langle\dcoh{[V_{\P_S(S^d\mathcal P^*)}(\pi^*\mathcal P)/\gm],w}, \mathcal A_1(1)\otimes\dbcoh{\P_S(S^d\mathcal P^*)},\ldots, \mathcal A_i(i)\otimes\dbcoh{\P_S(S^d\mathcal P^*)} \rangle.
\end{align*}
Observing that $V_{\P_S(S^d\mathcal P^*)}(\pi^*\mathcal P) \iso \op{V}_S(\mathcal P)\times_S\P_S(S^d\mathcal P^*)$, we obtain the required semi-orthogonal decomposition. 
\end{remark-section}

\begin{remark-section}
  If $d=2$, using the methods of \cite{BDFIKdegreed}, we recover Kuznetsov's construction for degree two Veronese embeddings \cite{Kuz05} (when $S$ is a point) and the relative version in \cite{ABB}.
\end{remark-section}

\end{document}